\documentclass[12pt]{amsart}
\usepackage[T1]{fontenc}
\usepackage[utf8]{inputenc}
\usepackage{amsmath,amssymb,amsthm,mathrsfs}
\usepackage{graphicx}
\usepackage{enumerate} 
\usepackage{lmodern}
\usepackage{mathrsfs} 
\usepackage{dsfont}
\usepackage{color}
\usepackage[all]{xy}
\usepackage{hyperref}
\usepackage[margin=3.5cm]{geometry}
\usepackage{caption}
\usepackage{subcaption}
\usepackage{algorithm2e}
\usepackage{lscape}
\usepackage{array}
\usepackage{adjustbox}

\usepackage{tikz}
\usepackage{tikz-cd}
\usepackage{comment}

\DeclareMathOperator{\Mod}{Mod}

\DeclareMathOperator{\im}{im}
\DeclareMathOperator{\upp}{up}

\begin{document}

\newcommand{\francois}[1]{{\color{red}Fran\c cois: #1}}
\newcommand{\steve}[1]{{\color{blue}Steve: #1}}
\newcommand{\alex}[1]{{\color{purple}Alex: #1}}

\newcommand{\mc}{ \mathcal}
\newcommand{\mr}{\mathrm}
\newcommand{\mb}{\mathbf}
\newcommand{\ms}{\mathscr}
\newcommand{\mf}{\mathfrak}
	
\newcommand{\lb}{[\![}
\newcommand{\rb}{]\!]}
\newcommand{\ra}{\rangle}
\newcommand{\la}{\langle}
	
\newcommand{\ioe}{\leqslant}
\newcommand{\soe}{\reslant}
	
\newcommand{\equi}{\Leftrightarrow}
\newcommand{\Equi}{\Longleftrightarrow}
\newcommand{\li}{\Leftarrow}
\newcommand{\ri}{\Rightarrow}
	
\newcommand{\A}{\forall}
\newcommand{\E}{\exists}
\newcommand{\st}{\subset}
	
\newcommand{\N}{\mathbb{N}}
\newcommand{\Z}{\mathbb{Z}}
\newcommand{\Q}{\mathbb{Q}}
\newcommand{\R}{\mathbb{R}}
\newcommand{\C}{\mathbb{C}}
	
\newcommand{\Bij}{\mathrm{Bij}}
\newcommand{\End}{\mathrm{End}}

\newcommand{\id}{\mathrm{id}}
\newcommand{\Iso}{\mathrm{Iso}}
\newcommand{\Aut}{\mathrm{Aut}}
\newcommand{\Ad}{\mathrm{ad}}
\newcommand{\ord}{\mathrm{ord}}
\renewcommand{\H}{\mathrm{H}}
\newcommand{\M}{\mathrm{M}}

\renewcommand{\P}{\mathbf{P}} 
\newcommand{\T}{\mathrm{T}} 
\renewcommand{\S}{\mathbf{S}} 

\newcommand{\K}{\mathbf{K}}
\newcommand{\Ch}{\mathbf{Ch}}
\newcommand{\D}{\mathbf{D}}
\newcommand{\V}{\mathbf{V}}
\newcommand{\W}{\mathbf{W}}
\newcommand{\Int}{\mathrm{int}}
\renewcommand{\k}{\mathbf{k}}
\newcommand{\kvg}{\k_{\V_\gamma}} 
\newcommand{\Dg}{\D_{\gamma^{\circ,a}}} 
\newcommand{\Modg}{\Mod_{\gamma^{\circ,a}}}
\newcommand{\SCov}{\mc{S}\mathrm{Cov}}
\newcommand{\Op}{\mathrm{Op}}
\newcommand{\Fun}{\mathrm{Fun}}
\newcommand{\op}{\mathrm{op}}
\newcommand{\Psh}{\mathbf{Psh}}
\newcommand{\raf}{\preceq}
\newcommand{\colim}{\mathrm{colim}}
\renewcommand{\a}{\mathfrak{a}}
\newcommand{\Eph}{\mathbf{Eph}}
\newcommand{\Pers}{\mathbf{Pers}}
\newcommand{\Ho}{\mathbf{Ho}}
\newcommand{\RR}{\mathrm{R}}
\newcommand{\stalk}{\mathrm{stalk}}
\newcommand{\dR}{\mathrm{R}}
\newcommand{\dL}{\mathrm{L}}
\newcommand{\sw}{\mathrm{sw}}
\newcommand{\HOM}[1][]{{\mathscr{H}\mspace{-4mu}om}_{\raise1.5ex\hbox to.1em{}#1}}
\newcommand{\B}{\mathbf{B}}
\newcommand{\Hom}[1][]{\mathrm{Hom}_{\raise1.5ex\hbox to.1em{}#1}}
\newcommand{\bb}{\mathbf{bb}}
\newcommand{\db}{\mathbf{db}}
\newcommand{\hb}{\mathbf{hb}}
\newcommand{\vb}{\mathbf{vb}}
\renewcommand{\Ho}{\ms{H}\mspace{-4mu}o}
\newcommand{\spn}{\mathrm{span}}
\newcommand{\coker}{\mathrm{coker}}
\newcommand{\dual}{\mathrm{D}}

\newcommand{\low}{\mathrm{low}}

\newcommand{\supp}{\mathrm{supp}}

\newcommand{\REF}{ {\color{blue} REF}}

\newcommand{\isoarrow}{\rotatebox{270}{\(\simeq\)}}
\newcommand{\ol}{\overline}

\newcommand{\PBT}{\mathbf{PBT}}

\newcommand{\family}{\mf{F}}

    \newcommand{\norm}[1]{\vert \vert #1 \vert \vert}
    \newcommand{\opnorm}[1]{ \vert \vert \vert #1 \vert \vert \vert}

\title[Computation of $\gamma$-linear projected barcodes]{Computation of $\gamma$-linear projected barcodes for multiparameter persistence}

\author{Alex Fernandes}
\address{\hspace{-0.25cm}Department of Mathematics, École normale supérieure, France}
\email{alex.fernandes@ens.fr}

\author{Steve Oudot}
\address{\hspace{-0.25cm} Inria Saclay -- \^Ile-de-France and \'Ecole Polytechnique, 91120, Palaiseau, France }
\email{steve.oudot@inria.fr}

\author{François Petit}
\address{\hspace{-0.25cm} Université Paris Cité and Université Sorbonne Paris Nord, Inserm, INRAE, Centre for Research in Epidemiology and Statistics (CRESS), 75004, Paris, France}
\email{francois.petit@inserm.fr}

\thanks{A.F was partially supported by the French Agence Nationale de la Recherche through the project reference ANR-22-CPJ1-0047-01 and by the Fondation Sciences Mathematiques de Paris through the Paris Graduate School for Mathematical Sciences scholarship.}

\thanks{F.P. was supported by the French Agence
Nationale de la Recherche through the project reference ANR-22-CPJ1-0047-01.}


	
	\theoremstyle{plain} 
	\newtheorem{theorem}{Theorem}[section]
	\newtheorem{corollary}[theorem]{Corollary}
	\newtheorem{proposition}[theorem]{Proposition}
	\newtheorem{lemma}[theorem]{Lemma}
	\theoremstyle{definition} 
	\newtheorem{definition}[theorem]{Definition}
	\newtheorem{example}[theorem]{Example}
	\newtheorem{remark}[theorem]{Remark}
	\newtheorem{examples}[theorem]{Examples}
	\newtheorem{question}[theorem]{Question}
	\newtheorem{Rem}[theorem]{Remark}
	\newtheorem{Notation}[theorem]{Notations}
	\newtheorem{conjecture}[theorem]{Conjecture}

\numberwithin{equation}{section}

\begin{abstract}
The $\gamma$-linear projected barcode was recently introduced as an alternative to the well-known fibered barcode for multiparameter persistence, in which restrictions of the modules to lines are replaced by pushforwards of the modules along linear forms in the polar of some fixed cone~$\gamma$. So far, the computation of the $\gamma$-linear projected barcode has only been studied in the functional setting, in which persistence modules come from the persistent cohomology of $\R^n$-valued functions. Here we develop a method that works in the algebraic setting directly, for any multiparameter persistence module over~$\R^n$ that is given via a finite free resolution. Our approach is similar to that of RIVET: first, it pre-processes the resolution to build an arrangement in the dual of~$\R^n$ and a barcode template in each face of the arrangement; second, given any query linear form~$u$ in the polar of~$\gamma$, it locates~$u$ within the arrangement to produce the corresponding barcode efficiently. While our theoretical complexity bounds are similar to the ones of RIVET, our arrangement turns out to be simpler thanks to the linear structure of the space of linear forms. Our theoretical analysis combines sheaf-theoretic and module-theoretic techniques, showing that multiparameter persistence modules can be converted into a special type of complexes of sheaves on vector spaces called {\em conic-complexes}, whose derived pushforwards by linear forms have predictable barcodes. 
\end{abstract}

\maketitle

\tableofcontents

\section{Introduction}

\subsection{Context}

Persistence theory studies the topological variations within a filtered family of topological spaces, and it encodes these variations in an algebraic object called a persistence module. When the filtered family of spaces is indexed over $\R$ (or, more generally, over a totally ordered set), the theory is well understood, and the resulting persistence modules admit a complete descriptor called the barcode. But when the indexing set of the family is $\R^n$ for some $n\geq 2$ (or, more generally, a partially ordered set), the situation is far more complex, and the direct generalization of the notion of barcode as the collection of indecomposable summands of the persistence modules is extremely complex which poses several pratical challenges. This has been a major hindrance to the the use of multiparameter persistence in applications.

Several approaches have been proposed to overcome this difficulty, in particular the resort to incomplete descriptors that are comparatively easy to compute and to interpret.   Among these, the {\em fibered barcode} \cite{CFFFL13} stands out as one that can be both efficiently computed and nicely represented (at least for 2-parameter persistence modules) thanks to the software RIVET~\cite{LW15}. The fibered barcode of a persistence module $M$ is defined as the collection of barcodes obtained by considering restrictions of $M$ to affine lines of positive slopes. As an incomplete descriptor, it cannot distinguish between certain pairs of non-isomorphic persistence modules, even some that are particularly simple~\cite{V20}. An approach to tackle this blind spot would be to enrich the construction of the fibered barcode with a larger class of operations than just restrictions. This is where the idea of turning persistence modules into sheaves comes into the picture, which has the great benefit of allowing then for the use of classical sheaf operations that are unavailable in the persistence modules framework. 

This was the idea pursued by the authors of \cite{BP23}, who introduced the concept of {\em projected barcode} of an $n$-parameter persistence module, defined as the collection of derived pushforwards---in a sheaf-theoretic sense---of the module along maps $\R^n\to\R$ from a prescribed family~$\family$. Interestingly, this concept encompasses the usual fibered barcode as a special case~\cite[Corollary 5.14]{BP23} and thus yields a strictly stronger descriptor if a large enough family~$\family$ is used. It also enjoys some stability properties akin to the ones enjoyed by the fibered barcode. In addition, despite its mathematically involved definition, the projected barcode along individual maps $u\colon \R^n\to\R$ turns out to be very simple to compute when the input module encodes the persistent homology of some $\R^n$-valued function~$f$: indeed, in this case, it is given essentially by the usual 1-parameter persistent cohomology of the composition $u\circ f$.  Thus, the sheaf-theoretic framework developed in~\cite{BP23} provides the algebraic foundation to an otherwise very simple and natural idea to cope with the multi-parameter persistent homology of $\R^n$-valued functions, which is to reduce the dimensionality of the problem via post-composition by multiple projections $\R^n\to\R$.

This framework, as appealing as it may sound,  remains nonetheless limited in two important ways:
(1) it is practical only in the functional setting, where derived pushforwards translate into post-compositions; (2) it only allows for the "pointwise" evaluation of the projected barcode, e.g., the computation of pushforwards along individual maps $\R^n\to\R$. It does not provide a description of the projected barcode as a whole over the family~$\family$. By contrast, the entire combinatorial structure of the usual fibered barcode can be encoded as a finite arrangement of hyperplanes augmented with barcode templates in its faces, as described in~\cite{LW15} and exploited first in RIVET then in subsequent work on computing the associated \emph{matching distance} \cite{BK21,KLO20}.
Our goal here is to lift these two limitations, largely taking inspiration from what has been done in RIVET for the fibered barcode.  

\subsection{Our contributions}

As our goal is arguably hard (if at all possible) to achieve for arbitrary families~$\family$ of maps~$\R^n\to\R$, we restrict our focus to a particular family, composed of linear forms of operator norm~$1$, more precisely, of those unitary linear forms that are located in the polar of some fixed cone~$\gamma$ in the dual of~$\R^n$. There are several good reasons for considering this particular family: first, the corresponding pushforwards can be defined in the derived category of $\gamma$-sheaves, which naturally connects to the category of persistence modules \cite{KS18, BP21}; second, the resulting projected barcode, called the {\em $\gamma$-linear projected barcode} , has a combinatorial structure that is simple enough to be encoded in a way similar to that of the fibered barcode; finally, the $\gamma$-linear projected barcode has been shown in~\cite{BP23} to nicely complement the fibered barcode, as each is able to discriminate between different sets of pairs of non-isomorphic persistence modules---see~\cite[Section 5.1]{BP23} for an example where the projected barcode is more discriminating, and for an example where the fibered barcode is more discriminating take an interval module supported on an infinite left-open vertical band.

At the heart of our contribution is an algorithm to pre-process a finitely presented $n$-parameter persistence module~$M$ (given through some pre-computed finite free resolution) into an augmented arrangement \`a la RIVET that encodes the combinatorial structure of its $\gamma$-linear projected barcode entirely.  This algorithm is completed by a routine that can efficiently compute the projected barcodes of~$M$ along individual query linear forms in~$\family$. A notable feature of our augmented arrangement is to be lower-dimensional than the one defined for the fibered barcode; this is particularly interesting in the 2-parameter case, where our arrangement becomes 1-dimensional and is therefore simpler to build and to query than the one used in RIVET in practice, even though the worst-cast complexity bounds are similar.
The details of our algorithm and of its associated query routine are given in Section~\ref{sec:algo_proj_barcode}. They rely on several new mathematical concepts and results, introduced in Sections~\ref{sec:conic-complexes} and~\ref{sec:extension_algo}:
\begin{itemize}
    \item the concept of {\em conic-complex} (Definition~\ref{def:conic-complex}), a special type of $\gamma$-complex of sheaves on~$\R^n$, in which the terms have a very simple structure akin to that of free persistence modules;
    \item the fact that the module-sheaf correspondence functor introduced in~\cite{BP21,KS18} sends the free resolution of our input persistence module to a conic-complex, and that a finitely presentable persistence module can be unambiguously recovered from its associated $\gamma$-sheaf (Proposition~\ref{prop_pullback_alpha},  Proposition~\ref{prop:ff_aalpha} and Proposition~\ref{prop:free_resol_to_conic_complex}~(2.));
    \item the fact that the derived pushforward along any linear form~$u$ sends this conic-complex on $\R^n$ to a conic-complex on $\R$, with a one-to one correspondence between the summands of the two complexes (Theorem~\ref{thm_proj_bar});
    \item the concept of {\em simplex-wise} filtered cochain complex (Definition~\ref{def:simplex-wise}), an axiomatization of the usual filtered simplicial cochain complexes arising in persistence, which allows for the use of matrix reduction in order to compute barcodes;
    \item the fact that our conic-complexes on~$\R$ can be turned into simplex-wise filtered cochain complexes, with a predictable effect on their barcode (Propositions~\ref{prop_finite_filtration} and~\ref{prop_finite_simplexwise_barcode}, Theorem~\ref{thm_red_qsw}). 
\end{itemize}
These new ingredients combine sheaf-theoretic and module-theoretic techniques, and down the road they produce an effective algorithmic way to enhance the popular software RIVET for multi-parameter topological data analysis. We believe that this interdisciplinarity is a notable aspect of our work.

In Section \ref{sec:example}, we detail a running example of our algorithm on a handcrafted persistence module and in Section~\ref{sec:experiments} we report some experimental results on point cloud data. The code is available \href{https://github.com/alexfrnds/conic-complex}{https://github.com/alexfrnds/conic-complex}.

\section{Preliminaries}\label{sec:prerequisite}

Let $\k$ be a field. We denote by $\Mod(\k)$ the category of $\k$-vector spaces. Let $X$ be a topological space, and let $\k_X$ be the constant sheaf on $X$, that is, the sheaf of locally constant functions on $X$. We write $\Mod(\k_X)$ for the category of sheaves of $\k$-vector spaces on $X$. Moreover, let $\Ch(\k)$ (resp. $\Ch(\k_X)$) be the category of cochain complexes of $\Mod(\k)$ (resp. $\Mod(\k_X)$). Following the notations of \cite{KS90}, let $\Ch^{\ast}(\k)$ (resp. $\Ch^{\ast}(\k_X))$ with $\ast=^b,^+,^-$ denote the full subcategories of bounded, bounded below, bounded above complexes, respectively. Finally, let $\D(\k_X)$ (resp.  $\D(\k)$) be the derived category of $\Mod(\k_X)$ (resp. $\Mod(\k)$), whose objects will merely be called sheaves. Again, $\D^\ast(\k_X)$ with $\ast=^b,^+,^-$ denotes the full subcategory of $\D(\k_X)$ spanned  by bounded, bounded below, bounded above objects of $\D(\k_X)$, respectively.

Throughout the paper, we use the six Grothendieck operations freely, with respective notations $\HOM$ (internal Hom), $\otimes$ (tensor product), $u_*$ (direct image or pushforward), $u_!$ (proper direct image), $u^{-1}$ (inverse image or pullback), $u^{!}$ (exceptional inverse image). We refer the reader to \cite[Chapters 2 and 3]{KS90} for definitions of these operations.


\subsection{Free resolutions of persistence modules}

Let $(\mc P, \leq)$ be a poset, viewed as a small thin category. 
\emph{Persistence modules} over~$\mc P$ are functors $\mc P \to \Mod(\k)$. They form a Grothendieck abelian category, denoted by $\Mod(\k)^{\mc P}$. Unless otherwise mentioned, all persistence modules will be assumed pointwise finite dimensional~(pfd).
A persistence module $M$ is \emph{free} if it admits a decomposition: 
\[
M \simeq \bigoplus_{b \in \mc M} \k^{\mr{up}(b)}, \ \text{ with } \k^{\mr{up}(b)}(x) := \begin{cases} \k \text{ if } b \leq x, \\ 0 \text{ otherwise} \end{cases}
\]
where $\mc M$ is a multiset of elements of $\mc P$ and the structural morphisms $\k \to \k$ in $\k^{\mr{up}(b)}$ are identities.
A \emph{free resolution} of a persistence module $M$ is the data of a cochain complex $L$ composed of free persistence modules and concentrated on non-positive degrees,  together with an augmentation morphism $L^0 \to M$, such that the following complex is exact: $\dots \to L^{-1} \to L^0 \to M \to 0 \to \cdots$.
It follows from Hilbert's syzygy theorem that every finitely presented persistence module over $\R^n$ admits finite free resolutions.


\subsection{\texorpdfstring{$\gamma$}{gamma}-topology and \texorpdfstring{$\gamma$}{gamma}-sheaves}

Let $\V$ be a finite-dimensional real vector space equipped with the Euclidean topology, and let $a: x \mapsto -x$ be the antipodal map. Following the notation of \cite{KS18}, for any subset $A \subset \V$ the antipodal of $A$ is denoted by $A^a := a(A) = -A$. Meanwhile,  $\Int(A)$ and $\overline{A}$ stand respectively  for the Euclidean interior and the Euclidean closure of $A$.

A \emph{cone} is a non-empty subset of $\V$ that is invariant by non-negative scaling. Its \emph{polar cone} is defined by: $C^\circ := \{ \eta \in \Hom(\V,\k) \mid \A c \in C : \eta(c) \geq 0 \}$. A cone $C$ is \emph{proper} if it is convex, pointed (i.e., $C^a \cap C = \{ 0 \}$) and solid (i.e., with non-empty interior).

From now on, let $\gamma$ be a proper cone. An open subset $\Omega$ of $\V$ is called $\gamma$\emph{-open} if it is $\gamma$\emph{-invariant}, i.e., $\Omega + \gamma = \Omega$. The $\gamma$\emph{-open} subsets of $\V$ form a topology on $\V$, called the \emph{$\gamma$-topology}, and $\V_\gamma$ stands for $\V$ equipped with this topology. Furthermore, the collection $ \{ x + \Int(\gamma) \}_{x \in \V}$ forms a basis for the $\gamma$-topology on $\V$. The continuous map $\phi_\gamma : \V \to \V_\gamma$ whose underlying application is the identity, yields an equivalence of triangulated categories \cite[Theorem 1.5]{KS18}: ${\phi_\gamma^{-1}: \D^b( \k_{\V_\gamma}) } \rightleftarrows { \Dg^b(\k_\V) :\dR {\phi_{\gamma}}_*}$, where $\Dg^{b}(\k_{\V})$ is the full subcategory of $\D^{b}(\k_{\V})$ consisting of objects with microsupport contained in $\V \times \gamma^{\circ,a}$ (We do not recall the notion of microsupport as it is neither used nor necessary to understand the present paper and refer the curious reader to \cite[Chapter5]{KS90}).


\subsection{From persistence modules to Alexandrov sheaves}\label{subsec:PM_to_Alex}


Let $(X, \leq)$ be a poset and let $X_\a$ denote the set $X$ equipped with the \emph{Alexandrov topology} on $(X,  \leq)$, i.e., the topology whose open sets are the $\leq$-lower closed sets $U\subseteq X$.
The choice of $\gamma$ induces a partial order on $\V$ given by $x \leq_{\gamma} y$ if $x + \gamma \subseteq y + \gamma$. Let $\V_{\a}$ stand for $\V$ equipped with the Alexandrov topology induced by $\leq_\gamma$, and let $\Mod(\k_{\V_{\a}})$ denote the category of sheaves of $\k$-vector spaces on $\V_\a$. The poset $(\V, \leq_\gamma)$ can also be equipped with the trivial Grothendieck topology turning it into the site denoted by $\V_{\leq_\gamma}$ (see \cite{KS6} for sheaves on Grothendieck topologies). Then,  $\Fun((\V, \leq_\gamma )^{\op}, \Mod(\k))$ is the category of sheaves over $\V_{\leq_\gamma}$ and is denoted by $\Mod(\V_{\leq_\gamma})$.
There is a morphism of sites $\theta : \V_\a \to \V_\leq$ given by the functor $\theta^t : x \mapsto x + \gamma$, which yields an equivalence of categories \cite[Theorem 4.2.10]{Curry}: 
\[{\theta_*: \Mod(\k_{\V_\a}) } \rightleftarrows { \Mod(\V_{\leq_\gamma}) :\theta^{-1}}.\]
 Notice also that  $(\V, \leq_\gamma )^{\op}$ is equivalent to  $\V$ endowed with the opposite order $\leq_\gamma^\op$. This yields trivially an equivalence $ \Mod(\V_{\leq_\gamma}) \simeq \Fun((\V,\leq_\gamma^\op),\Mod(\k))$. When $\V=\R^n$ and $\gamma=[0, + \infty)^n$, the order $\leq_\gamma^\op$ corresponds to the order product on $\R^n$, denoted by $\leq$, and then $\Fun((\V,\leq_\gamma^\op),\Mod(\k))$ is what is usually called the category of persistence modules on $\R^n$ (not necessarily pfd). We will use interchangeably the term persistence module to designate the objects of $\Mod(\V_{\leq_\gamma})$ and of $\Fun((\V,\leq_\gamma^\op),\Mod(\k))$. We say that $F \in \Mod(\k_{V_\a})$ is \textit{finitely presentable} if the persistence module $\theta_\ast F$ is finitely presentable.
%
%
%


\subsection{\texorpdfstring{$\gamma$}{gamma}-sheaves and Alexandrov sheaves}

In \cite{BP21}, the authors define two morphisms of sites: $\alpha:\V_\gamma \to \V_\a$ and $\beta: \V_\a \to \V_\gamma$, given respectively by the following functors:
\begin{flalign*}
\alpha^t &: \Op(\V_\a)  \to  \Op(\V_\gamma) \quad \quad x + \gamma \mapsto  x + \Int(\gamma),\\
 \beta^t &: \Op(\V_\gamma) \to \Op(\V_\a) \quad \quad x + \Int(\gamma) \mapsto x + \Int(\gamma). 
\end{flalign*}
These morphisms of sites establish a link between $\gamma$-sheaves and Alexandrov sheaves through the following two adjunctions:
\begin{align*}
	{\alpha^{-1}: \Mod(\k_{\V_\a}) } &\rightleftarrows { \Mod(\k_{\V_\gamma}) : \dR \alpha_*},\\
	{\beta^{-1}: \Mod(\k_{\V_\gamma}) } &\rightleftarrows{\Mod(\k_{\V_\a}) :\beta_*}.
\end{align*}
The functors $\alpha^{-1}$, $\beta_*$ are isomorphic to each other, and $\alpha_*$, $\beta^{-1}$ are fully faithful. These functors can be derived which yields
\begin{equation*}
\begin{tikzcd}
\beta_*=\alpha^{-1} \colon \D(\k_{\V_\a}) \arrow[r]  & \D(\k_{\V_\gamma})  \arrow[l, shift right=1.5] \arrow[l, shift left=1.5] \colon \dR\alpha_*, \beta^{-1}.
\end{tikzcd}
\end{equation*}
The functors $\dR\alpha_*$ and $\beta^{-1}$ are fully faithful.


\subsection{One-dimensional persistence modules and \texorpdfstring{$\gamma$}{gamma}-linear projected barcode}


We denote by $\D_{\R c}^b(\k_\V)$ the full triangulated subcategory of $\D^b(\k_\V)$ consisting of  objects $\ms F$ whose cohomology groups $\H^k(\ms F)$, $k\in\Z$, are $\R$-constructible sheaves. A reference on constructible sheaves is \cite[section VIII.8.4]{KS90}. 

From now on, we assume that the cone $\gamma$ is subanalytic in $\V$. Let $\D^b_{\R c}(\k_{\V_{\infty}})$ stand for the full triangulated subcategory of $\D^b_{\R c}(\k_\V)$ consisting of sheaves \emph{constructible up to infinity}, that is: sheaves whose microsupport is subanalytic in $\T^*\P$, where $j:\V \to \P=\P^n(\V \oplus \R), j(x) := [x:1]$ is the \emph{projectivization} of $\V$ (see \cite{S23} for an extensive exposition on constructible sheaves up to infinity).

We now assume that $\V$ is endowed with a subanalytic norm $\norm{\cdot}$ \cite{BM88}, and we let $\S^*$ be the unit sphere in $\V^*$ equipped with the operator norm associated to~$\norm{\cdot}$. In \cite[section 5]{BP23}, the authors introduce the $\gamma$\emph{-linear projected barcode} as the functor:
\begin{align}\label{map:projbar}
\ms P^\gamma : \Int(\gamma^{\circ}) \cap \S^* \times \D^b_{\R c, \gamma^{\circ, a}}(\k_{\V_\infty})  &\to \D^b_{\R c, [0, +\infty)^{\circ, a}}(\k_{\R_\infty})  
             &  \hspace{-0.2mm}(u,\ms F)\mapsto  \dR u_! \ms F.
\end{align}

\begin{remark}
 \begin{enumerate}[(i)]
  \item In the definition of the projected barcode, we only consider the pushforward of a sheaf by linear forms. Indeed the pushforward by an affine map can be decomposed as the pushforward by a linear map followed by the pushforward by a translation. The pushforward by the translation does not alter the structure of the barcode, as it merely shifts the barcode. This differs from the situation encountered with the fibered barcode. The fibered barcode is defined as the pullback of sheaves or persistence modules along affine maps. Although it is possible to restrict to pullbacks by linear forms, it is not possible to reconstruct the pullbacks by affine maps from the linear ones. This is a major difference between the projected and fibered barcodes.

    \item In the definition of the $\gamma$-linear projected barcode, it is sufficient to only consider linear forms in $\Int(\gamma^\circ)$. This comes from \cite[Proposition 4.5]{BP23}, which shows that ignoring forms in  $ \V^\ast \setminus \left( \Int(\gamma^{\circ}) \cup \Int(\gamma^{\circ,a}) \right)$ is harmless as these forms do not bring any information.  Finally, since the barcodes produced by two colinear non-zero linear forms differ by a homothety it is possible to restrict to $\Int(\gamma^{\circ}) \cap \S^\ast$ without loosing any information.
\end{enumerate}
\end{remark}

%
%
The projected barcode takes advantage of the following facts.
\begin{theorem}[{\cite[Thm.~1.17]{KS18}}]\label{thm:Decomposition}
For any sheaf $\ms F$ in $\Mod_{\R_c}(\k_\R)$, there exists a unique locally finite multiset $\B(\ms F)$ of intervals of $\R$, called the \emph{barcode} of $\ms F$, such that:
$\ms F \simeq \bigoplus_{I \in \B(\ms F)} \k_I$.
\end{theorem}

The functor $\alpha^{-1}\theta^{-1}$ associates to any persistence module $M$  a sheaf $\alpha^{-1} \theta^{-1} M$. On $\R$, if $M$ is pfd then $\phi_\gamma^{-1}\alpha^{-1} \theta^{-1} M$ is constructible and therefore (by Theorem~\ref{thm:Decomposition}) admits a barcode, called the \emph{observable barcode} and denoted by $\ol{\B}(M)$. Note that this barcode  coincides with the barcode of~$M$ in the observable category~\cite{CCBS16}, not with the usual barcode $\B(M)$ coming from $M$'s direct-sum decomposition~\cite{C15}.
 
A sheaf $\ms F$ in $\D^b_{\R_c}(\k_\R)$ is isomorphic to the direct sum of its shifted cohomology groups, i.e.:
$ \ms F \simeq \bigoplus_{i \in \Z} \H^i(\ms F)[-i].$
In particular, the \emph{graded barcode} $\B(\ms F) := \{ \B^d(\ms F) := \B(\H^j(\ms F)) \}_{j \in \Z}$ introduced in \cite{BG22} is a complete discrete invariant of the isomorphism class of $\ms F$.


\section{Conic complexes and their basic properties}\label{sec:conic-complexes}
\subsection{Definitions and basic properties}

Conic-complexes are a class of sheaves that arise naturally from free resolutions of persistence modules. Here we define conic-complexes and derive some of their basic properties.

\begin{definition}
A \emph{closed (resp. open) $\gamma$-complex} is an object of $\Ch^+(\k_\V)$ (resp. $\Ch^-(\k_\V)$) denoted $\gamma\{J\}$ (resp. $\Int(\gamma)\{J\}$) of the form: 
\begin{flalign*}
\gamma\{J\} &:  \quad 0 \to \bigoplus_{j_{0} \in J_0} \k_{j_{0} + \gamma} \to \bigoplus_{j_{1} \in J_1} \k_{j_{1} + \gamma} \to  \bigoplus_{j_{2} \in J_2} \k_{j_{2} + \gamma} \to \cdots & \\
\Int(\gamma)\{J\} &: \quad \dots \to \bigoplus_{j_{1} \in J_1} \k_{j_{1} + \Int(\gamma)} \to \bigoplus_{j_{0} \in J_0} \k_{j_{0} + \Int(\gamma)} \to 0 &
\end{flalign*}
where the $J_i$ for $i\in\N$ are multisets of $\V$. The elements of the multiset $J=\bigcup_{i=0}^{+\infty} J_i$ are called \emph{generators}.
\end{definition}

\begin{definition}\label{def:conic-complex}
A \emph{conic-complex} refers to a closed or open $\gamma$-complex. A conic-complex is \emph{bounded} if it is a bounded cochain complex, \emph{locally finite} if for any $i \in \N$ the multiset $J_i$ is finite, and \emph{finite} if it is both locally finite and bounded.
\end{definition}

\begin{definition} Let $\ms F$ be a sheaf in $\Mod( \k_\V)$. An open $\gamma$-resolution of $\ms F$ is the data of an open $\gamma$-complex $\Int(\gamma)\{J\}$ with an augmentation morphism $\eta : (\Int(\gamma)\{J\})^0 \to \ms F$ such that the following complex of $\Mod( \k_\V)$ is exact: 
\[ \dots \to  \bigoplus_{j_{2} \in J_2} \k_{j_{2} + \Int(\gamma)} \to \bigoplus_{j_{1} \in J_1} \k_{j_{1} + \Int(\gamma)} \to \bigoplus_{j_{0} \in J_0} \k_{j_{0} + \Int(\gamma)} \overset{\eta}{\to} \ms F \to 0 .\]
\end{definition}


We will use duality for sheaves. Based on the notation of \cite[ Definition 3.1.16]{KS90}, the dualising complex is written $\omega_\V := a_\V^! \k$ for $a_\V: \V \to *$, and for any sheaf $\ms F$ in $\D^b(\k_\V)$
\begin{equation*} 
\dual_\V \ms F := \dR \HOM ( \ms F, \omega_\V ) \quad \quad \dual'_\V \ms F := \dR \HOM (\ms F, \k_\V).
\end{equation*}

\begin{lemma}\label{lem_dual_calcul} Let $b \in \V$. Then, $\dual'_\V \k_{b + \gamma} \simeq \k_{b + \Int(\gamma)}$ and $\dual'_\V \k_{b + \Int(\gamma)} \simeq \k_{b + \gamma} $.
\end{lemma}

\begin{proof} The statement follows from \cite[Exercise III.4]{KS90} after observing that $b + \gamma$ is convex hence locally cohomologically trivial. We refer the reader to Appendix \ref{app:duality} for a brief overview and precise reference on the notion locally cohomologically trivial open sets. 
\end{proof}

\begin{proposition} \label{prop_dual_resolution}
The following isomorphisms hold in $\D^b(\k_\V)$:
\[ \dual'_\V(\gamma\{J\}) \simeq \Int(\gamma)\{J\}, \quad \quad \dual'_\V(\Int(\gamma)\{J\}) \simeq \gamma\{J\}. \]
\end{proposition}
\begin{proof}
The functor $\dual'_\V$ preserves finite direct sums, so for any multiset $J$ it follows from Lemma \ref{lem_dual_calcul} that:
\[ \dual'_\V \big( \bigoplus_{j \in J} \k_{j+\gamma} \big) \simeq \bigoplus_{j \in J} \dual'_\V (  \k_{j+\gamma}) \simeq \bigoplus_{j \in J} \k_{j+ \Int(\gamma)}  .\]
In particular, for any integer $i$, the object $\gamma\{J\}^i$ is $\HOM(-,\k_\V)$-acyclic, therefore the statement follows from Lemma \ref{lem_acyclic}. This proof works for the second statement mutatis-mutandis
\end{proof}

\begin{proposition}\label{prop_constr}
Let $\ms L$ be a finite conic-complex. Then,  $\ms L$ is constructible up to infinity. In particular, if $\ms L$ is a finite conic-complex on $\R$, the multiset $\B^k(\ms L)$ is finite for every $k\in\Z$.
\end{proposition}

\begin{proof}
Since $b+\gamma$ is subanalytic for any $b \in \V$, it follows from \cite[Definition 8.4.3 \& Theorem 8.4.2 ]{KS90}  that the sheaf $\k_{b + \gamma}$ is $\R$-constructible. Moreover, it follows from \cite[Lemma 2.19]{S23} that $b + \gamma$ is subanalytic up to infinity, and given the projectivization map $j : \V \rightarrowtail \P$, we have $j_!\k_{b+\gamma} \simeq \k_{b+\gamma}$. Therefore, the sheaf $\k_{b+\gamma}$ is subanalytic up to infinity, by \cite[Lemma 2.7]{S23}. This implies that $\ms L$ itself is constructible up to infinity. Then, by \cite[Lemma 2.7 (c)]{S23}, the multiset $\B^k(\ms L)$ is finite for every $k\in\Z$.
\end{proof}

\begin{lemma} \label{lemma_RHom_constant}
Let $b,c \in \V$. The following isomorphisms hold in $\D^b(\k_\V)$:
\begin{align*}
\dR \HOM (\k_{b + \gamma} , \k_{c + \gamma}) &\simeq \dual'_\V ( \k_{b + \gamma \cap c + \Int(\gamma)}),& \\
\dR \HOM (\k_{b + \Int(\gamma)} , \k_{c + \Int(\gamma)}) &\simeq \dual'_\V  ( \k_{c + \gamma \cap b + \Int(\gamma)}).&
\end{align*}
\end{lemma}
\begin{proof}
It follows from \cite[Proposition 2.6.3]{KS90} that:
\begin{align*}
\dR \HOM (\k_{b + \gamma} , \k_{c + \gamma})  &\simeq \dR \HOM (\k_{b + \gamma} \otimes \k_{c + \Int(\gamma)}, \k_\V) & \\ 
& \simeq \dR \HOM (\k_{b + \gamma \cap c + \Int(\gamma)}, \k_\V) & \\
& \simeq \dual'_\V(\k_{b + \gamma \cap c + \Int(\gamma)}).
\end{align*}
The proof of the second statement is similar and therefore omitted. 
\end{proof}

\begin{proposition}\label{prop_morphism}
Let $b,c \in \V$. Then, the following isomorphisms holds in  $\D^b(\k)$: 
\[\dR \Hom (\k_{b + \Int(\gamma)}, \k_{c + \Int(\gamma)}) \simeq \begin{cases} \k \text{ if }  b \leq_\gamma c, \\ 0 \text{ otherwise} \end{cases} \hspace{-4.9pt} \dR\Hom (\k_{b + \gamma}, \k_{c + \gamma}) \simeq \begin{cases} \k \text{ if }  b \geq_\gamma c, \\ 0 \text{ otherwise}. \end{cases} \]
\end{proposition}

\begin{proof}
Let $b$ and $c \in \V$. We first notice that the closed case follows from the open case. Indeed,
\[ \dR \HOM (\k_{c + \gamma} , \k_{b + \gamma}) \simeq \dual'_\V ( \k_{c + \gamma \cap b + \Int(\gamma)}) \simeq \dR \HOM (\k_{b + \Int(\gamma)} , \k_{c + \Int(\gamma)}).\]
Applying the global section functor to the above isomorphism yields
\begin{equation*}
\dR \Hom (\k_{c + \gamma} , \k_{b + \gamma})  \simeq \dR \Hom(\k_{b + \Int(\gamma)} , \k_{c + \Int(\gamma)}).
\end{equation*}
Hence, we only prove the open case. We distinguish two cases.\\
Assume that $b + \Int(\gamma) \subset c + \Int(\gamma)$. Then,
\begin{align*}
    \dR \Hom(\k_{b+\Int(\gamma)},\k_{c+\Int(\gamma)}) &\simeq \dR \Hom(\k_{b+\Int(\gamma)},\k_{c+\Int(\gamma)}\vert_{b+\Int(\gamma)})\\
    & \simeq\dR \Gamma(b+\Int(\gamma),\k_{b+\Int(\gamma)})
\end{align*}
Since $b+\Int(\gamma)$ is contractible, it follows that $\dR \Gamma(b+\Int(\gamma),\k_{b+\Int(\gamma)}) \simeq \k$. Hence, $\dR \Hom(\k_{b+\Int(\gamma)},\k_{c+\Int(\gamma)}) \simeq \k$.

\noindent(2) Assume that $b + \Int(\gamma)$ is not included in $c + \Int(\gamma)$.

Let $w \in \Int(\gamma)$ and consider the convex open set $B=-w+\Int(\gamma) \cap w + \Int(\gamma)^a$. By construction, we have $B+\gamma=\Int(\gamma)$. Since $b + \Int(\gamma)$ is not included in $c + \Int(\gamma)$, we can further assume, up to rescaling $B$, that $b + B \cap c + \Int(\gamma) = \emptyset$. We have 
\begin{align*}
\dR \Hom(\k_{b+\Int(\gamma)},\k_{c+\Int(\gamma)}) &\simeq \dR \Gamma(b+ \Int(\gamma), \k_{c+\Int(\gamma)})\\
& \simeq \dR \Gamma (b+B+\gamma,\k_{c+\Int(\gamma)})\\
& \simeq \dR \Gamma (b+B,\k_{c+\Int(\gamma)}) \quad\textnormal{\cite[Prop. 3.5.3 (i)]{KS90}}\\
&\simeq \dR \Hom(\k_{b+B},\k_{c+\Int(\gamma)} \vert_{{b+B}})\\
& \simeq 0.
\end{align*}
\end{proof}
%
%
%
\section{Multiparameter conic complexes from free resolutions}

If $X$ is a topological space, then for any locally closed subset $A \subseteq X$ we denote by $\k_{A}$ the constant sheaf on $X$ with stalk $\k$ on $A$ and $0$ elsewhere. If $X=\V$, then, to avoid confusion, we write $\k_{A}$ (resp. $\k_{A_\gamma}$, $\k_{A_\mathfrak{a}}$) when $A$ is considered as a locally closed subset of $\V$ (resp. $\V_\gamma$, $\V_\mathfrak{a}$). Let $A, \Omega$ and $F$ be respectively locally closed, open, and closed subsets of  $\V_\mathfrak{a}$. 
%
%
%

\begin{proposition}
\label{prop_pullback_alpha}
The following identity holds: $\alpha^{-1}\k_{\Omega \cap F}\simeq \k_{(\Int(\Omega) \cap \overline{ F})_\gamma}$.
\end{proposition}

\begin{proof}
Since inverse images are monoidal, it is sufficient to show that \[ \alpha^{-1} \k_\Omega \simeq \k_{\Int(\Omega)_\gamma}, \quad \quad  \alpha^{-1} \k_F \simeq \k_{(\overline F)_\gamma}.\]
We start by proving the first formula. For any open set  $U$ of a topological space $X$, we write $i_U$ for the natural morphism of presites $U \to X$. We denote by $\alpha_\Omega \colon \Int(\Omega) \to \Omega$ the morphism of sites induced by $\alpha:\V_\gamma \to \V_\a$. It is well-defined since $\alpha^t(\Omega)=\Int(\Omega)$. We now consider the following commutative diagram of morphisms of sites:
\[\begin{tikzcd}
	{ \Int(\Omega)} & {\V_\gamma} \\
	{\Omega} & {\V_\a}
	\arrow["{i_{\Int(\Omega)}}", from=1-1, to=1-2]
	\arrow["{\alpha_{\Omega}}"', from=1-1, to=2-1]
	\arrow["{i_{\Omega}}"', from=2-1, to=2-2]
	\arrow["\alpha", from=1-2, to=2-2]
\end{tikzcd}\]
It follows from \cite[Proposition 17.6.7]{KS6} that:
\begin{flalign*}
 \alpha^{-1}\k_{\Omega} & \simeq\alpha^{-1} {(i_{\Omega})_!} {(i_{\Omega})}^{-1} \k_{\V_\a}  \\ 
    &\simeq{(i_{\alpha^t(\Omega)})_!} (\alpha_{\Omega})^{-1} (i_{\Omega})^{-1} \k_{\V_\a} \\
    &\simeq{(i_{\Int(\Omega)})_!}(i_{\Omega} \circ  \alpha_{\Omega})^{-1} \k_{\V_\a} \\
    & \simeq{(i_{\Int(\Omega)})_!} (\alpha \circ i_{\Int(\Omega)})^{-1} \k_{\V_\a} \\
    & \simeq{(i_{\Int(\Omega)})_!} (i_{\Int(\Omega)})^{-1} \alpha^{-1} \k_{\V_\a} \\
    & \simeq \k_{\Int(\Omega)}.
\end{flalign*}
Let $U := \V_\a \setminus F$ and consider the following short exact sequence \cite[Proposition 2.3.6]{KS90}:
\[\begin{tikzcd}
	0 & \k_U & \k_{\V_\a} & \k_{F}& 0.
	\arrow[from=1-1, to=1-2]
	\arrow[from=1-2, to=1-3]
	\arrow[from=1-3, to=1-4]
	\arrow[from=1-4, to=1-5]
\end{tikzcd}\]
Using the isomorphisms $\alpha^{-1} \k_U \simeq \k_{\Int(U)}$ and $\alpha^{-1}\k_{\V_\a} \simeq \k_{\V_\gamma}$ coming from the preceding argument, and the exactness of $\alpha^{-1}$, one obtains the following short exact sequence:
\[\begin{tikzcd}
	0 & \k_{\Int(U)} & \k_{\V_\gamma} & \alpha^{-1} \k_{ F } & 0 
	\arrow[from=1-1, to=1-2]
	\arrow["\rho", from=1-2, to=1-3]
	\arrow[from=1-3, to=1-4]
	\arrow[from=1-4, to=1-5].
\end{tikzcd}\]
Observe that $\rho|_{\overline F}$ vanishes so $\alpha^{-1} \k_{F}|_{\overline F} \simeq \k_{\overline F}$ and $\rho|_{\Int(U)}$ induces an isomorphism on stalks, therefore $\alpha^{-1} \k_{F}|_{\Int(U)} \simeq 0$. It follows from \cite[Proposition 2.3.6 (i)]{KS90} that $\alpha^{-1} \k_{F} = \k_{\overline F}$.
\end{proof}

Let $b \in \V$. The unit $\eta \colon \dR \alpha_\ast \alpha^{-1} \to \id$ of the adjunction $(\alpha^{-1}, \dR \alpha_\ast)$ induces a morphism
    \begin{equation}\label{map:gamma_to_alex_cone}
    \begin{tikzcd}
       \eta_{(b+\gamma)} \colon \dR \alpha_\ast \alpha^{-1} \k_{(b+\gamma)_\a} \arrow[r]   & \k_{(b+\gamma)_\a} \\
    \end{tikzcd}
    \end{equation}
\begin{proposition} \label{prop:acyclicity_free}
\begin{enumerate}[(i)]
\item The morphism \eqref{map:gamma_to_alex_cone} is an isomorphism,
\item $\dR \alpha_\ast \k_{(b+\Int(\gamma))_\gamma} \simeq \k_{{(b+\gamma)}_\a}$.
\end{enumerate}
\end{proposition}

\begin{proof}
\noindent (i) Let us check that  the map \eqref{map:gamma_to_alex_cone} is an isomorphism by checking that  for every $z \in \V_\a$ its stalk at $z$ is an isomorphism. Taking the stalk in $z$ yields the following commutative diagram
    \begin{equation*}
      \begin{tikzcd}
         \left(\dR\alpha_\ast \alpha^{-1} \k_{(b+\gamma)_\a}\right)_z \arrow[r] \arrow[d,"\wr"'] & \k_{(b+\gamma)_\a,z} \arrow[d,"\wr"] \\
         \dR\Gamma(z+\Int(\gamma),\k_{(b+\Int(\gamma))_\gamma}) \arrow[r]& \k_{(b+\gamma)_\a,z}.
      \end{tikzcd}    
    \end{equation*}
If $z \in b+\gamma$, then using Proposition \ref{prop_morphism}, the bottom arrow of the above diagram reduces to an non-trivial linear map $\k \to \k$ and hence is an isomorphism.
If $z \notin b+\gamma$, then using Proposition \ref{prop_morphism} this same arrow reduces to a map from zero to zero and hence is an isomorphism. This proves that the map \eqref{map:gamma_to_alex_cone} is an isomorphism

\noindent (ii) It follows from Proposition \ref{prop_pullback_alpha} that $\alpha^{-1} \k_{(b+\gamma)_\a}\simeq \k_{(b+\Int(\gamma))_\gamma}$. Then the result follows from (i).
\end{proof}

\begin{proposition} \label{prop:ff_aalpha}
    Let $F \in \Mod(\k_{V_\a})$ be finitely presentable. Then \[ \dR \alpha_\ast \alpha^{-1}F \simeq F.\]
\end{proposition}

\begin{proof}
   Since $F$ is finitely presentable, it follows from Hilbert’s syzygy theorem that $F$ admits a finite resolution
   $ L^\bullet : 0 \to L^{-n} \to \dots \to L^{-1} \to L^0  \to 0 $ with $L^i = \bigoplus_{k \in I_i} \k_{(b_k + \gamma)_\a}$ and $I_i$ finite. Moreover, it follows from Proposition \ref{prop:acyclicity_free} that $\dR \alpha_\ast \alpha^{-1} L^i \simeq \alpha_\ast \alpha^{-1} L^i \simeq  L^i$. This, together with Lemma \ref{lem_acyclic} implies that $\dR\alpha_\ast \alpha^{-1} F \simeq \alpha_\ast \alpha^{-1} L^\bullet$ in $\D(\k_{\V_\a})$. The natural transformation $\eta$ induces the following commutative diagram  
   \begin{equation*}
   \begin{tikzcd}
    0 \arrow[r] \ar[d,"\wr"] & \alpha_\ast \alpha^{-1} L^{-n} \arrow[r] \ar[d,"\wr"]& \dots \arrow[r]& \alpha_\ast \alpha^{-1} L^{-1} \arrow[r] \ar[d,"\wr"]& \alpha_\ast \alpha^{-1} L^0 \arrow[r] \ar[d,"\wr"]& 0 \ar[d,"\wr"] \\
    0 \arrow[r]  &  L^{-n} \arrow[r] & \dots \arrow[r]& L^{-1} \arrow[r] &  L^0 \arrow[r] & 0 
   \end{tikzcd}
   \end{equation*}
   Proposition \ref{prop:acyclicity_free} implies that the vertical maps in the above diagram are isomorphisms. Thus,  $\alpha_\ast \alpha^{-1} L^\bullet \simeq L^\bullet$ which yields $\dR\alpha_\ast \alpha^{-1} F \simeq F$.
\end{proof}

\begin{remark}
    It follows from the proof of Proposition \ref{prop:acyclicity_free}, that the restriction of the functor $\alpha^{-1}$ to the full subcategory of $\D(\k_{\V_\a})$ generated by the objects isomorphic to bounded complexes of the form $0 \to L^{p} \to \dots \to L^{q} \to 0$ with $p, q \in \Z$, $L^i = \bigoplus_{k \in I_i} \k_{(b_k + \gamma)_\a}$ and  $I_i$ finite, is fully faithful.
\end{remark}

\begin{definition}
Given a subset $A$ of $\V$, the \emph{indicator  module supported on $A$}, denoted by $\k^A$, is defined as follows, where all the structural morphisms $\k\to\k$ are identities:
\begin{equation*}
\k^A (x) := \begin{cases} \k \text{ if } x \in A, \\ 0 \text{ otherwise}. \end{cases}
\end{equation*}
\end{definition}

The aim of the next lemma is to record a few well-known facts that are of constant use in this paper.

\begin{lemma}\label{prop:free_resol_to_conic_complex}
Let $b, c \in \V$. Then, we have
    \begin{enumerate}[(i)]
        \item $\theta_\ast \k_{A_\mathfrak{a}} \simeq \k^A$, 
        
        \item $\theta^{-1} \k^{\mr{up}(b)} \simeq \k_{(b + \gamma)_\a}$,

        \item $ \alpha^{-1}  \k_{(b + \gamma)_\a} \simeq \k_{(b + \Int(\gamma))_\gamma}$,

        \item $\phi^{-1}_\gamma \k_{(b + \Int(\gamma))_\gamma} \simeq \k_{b + \Int(\gamma)}$,
        \item The following diagram is commutative
        \begin{equation}\label{diag:morphism_alex_gamma}
          \begin{tikzcd}[column sep=large]
              \Hom(\k^{\mr{up}(b)},\k^{\mr{up}(c)}) \arrow[r,"\sim","\phi^{-1}_\gamma \alpha^{-1} \theta^{-1}"'] \arrow[d]
              & \Hom(\k_{b + \Int(\gamma)}, \k_{c + \Int(\gamma)}) \arrow[d] \\
               \k \ar[r,"\id"] 
               &\k
          \end{tikzcd}
        \end{equation}
    and the map $\phi^{-1}_\gamma \alpha^{-1} \theta^{-1}$ is an isomorphism.
    \end{enumerate}
\end{lemma}
\begin{proof}
\begin{enumerate}[(i)]
\item Follows from the definition of $\theta_\ast$,
\item Since $\theta_\ast$ is an equivalence, we have $\theta^{-1}\k^A\simeq \k_{A_\mathfrak{a}}$,
\item Follows from Proposition \ref{prop_pullback_alpha},
\item Follows from \cite[Eq. (2.3.19)]{KS90},
\item The functor $\theta^{-1}$ is an equivalence of category and $\phi_\gamma^{-1}$ is fully faithful. Moreover, it follows from Proposition \ref{prop:ff_aalpha} that the map
\begin{equation*}
    \Hom(\k_{(b + \gamma)_\a},\k_{(c + \gamma)_\a}) \stackrel{\alpha^{-1}}{\longrightarrow} \Hom(\k_{(b + \Int(\gamma))_\gamma}, \k_{(c + \Int(\gamma))_\gamma})
\end{equation*}
is an isomorphism. Hence, the map $\phi^{-1}_\gamma \alpha^{-1} \theta^{-1} $ is an isomorphism.\\
We prove the commutativity of the diagram \eqref{diag:morphism_alex_gamma}. First, if $b+\gamma$ is not included in $c+\gamma$, the commutativity follows from 
\[
\Hom(\k_{(b + \gamma)_\a},\k_{(c + \gamma)_\a}) \simeq \Hom((\k_{b + \Int(\gamma))_\gamma}, \k_{(c + \Int(\gamma))_\gamma}) \simeq 0.
\]
Now assume that $b+\gamma \subset c+\gamma$. Then
\[
\Hom(\k_{(b + \gamma)_\a},\k_{(c + \gamma)_\a}) \simeq \Hom((\k_{b + \Int(\gamma))_\gamma}, \k_{(c + \Int(\gamma))_\gamma}) \simeq \k
\]
and these two isomorphisms are provided respectively by the morphisms
\begin{align*}
\Hom(\k_{(b + \gamma)_\a},\k_{(c + \gamma)_\a}) \to \k,& \hspace{0.3cm} \psi \mapsto \psi_{b+\gamma}({\bf 1}_{b+\gamma})\\  
\Hom(\k_{(b + \Int(\gamma))_\gamma},\k_{(c + \Int(\gamma))_\gamma}) \to \k,& \hspace{0.3cm} \psi \mapsto \psi_{b+\Int(\gamma)}({\bf 1}_{b+\Int(\gamma)})
\end{align*}
where ${\bf 1}_{b+\gamma}$ and ${\bf 1}_{b+\Int(\gamma)}$ are respectively the indicator functions of $b+\gamma$ and $b+\Int(\gamma)$. The commutativity follows immediately. 
\end{enumerate}
\end{proof}

\begin{corollary}\label{prop:fpisconstr}
    Let $M$ be a finitely presentable persistence module on $\V$. Then the sheaf $\phi_\gamma^{-1} \alpha^{-1} \theta^{-1} M$ is constructible up to infinity 
\end{corollary}

\begin{proof}
  Since $M$ is finitely presentable, it admits a free resolution. It follows from Lemma \ref{prop:free_resol_to_conic_complex}, that $\phi_\gamma^{-1} \alpha^{-1} \theta^{-1} M$ admits a resolution by a finite conic complex. Hence, it is constructible up to infinity by Proposition \ref{prop_constr}.
\end{proof}

In view of the Proposition \ref{prop:fpisconstr}, we make the following definition 

\begin{definition}
The projected barcode of a finitely presentable persistence module $M$ is the projected barcode of the associated sheaf $\phi_\gamma^{-1} \alpha^{-1} \theta^{-1} M$.
\end{definition}

\section{Direct images of finite conic-complexes}\label{subsec:direct_image}

Computing the direct image of a finite conic-complex by a linear form  is the first step of the computation of the projected barcode of this conic-complex. Here, a theoretical description of the derived proper direct image by a linear form of a conic complex is provided (Theorem~\ref{thm_proj_bar}).

Let $u \in \Int(\gamma^{\circ})$ and $v \in \Int(\gamma^{\circ,a})$. It follows from \cite[Lemma 4.7]{BP23} that, if $K \subset \V$ is compact, then $u|_{K+\gamma}$ and $v|_{K+\gamma^a}$ are proper. In particular, $u_*\k_{K+\gamma} \simeq u_!\k_{K+\gamma}$ and $v_*\k_{K+\gamma^a} \simeq v_!\k_{K+\gamma^a}.$ Recall that $\V$ is $n$-dimensional. Let $\Lambda := [0, +\infty)$. 

\begin{theorem}\label{thm_proj_bar}
Let $\ms Z:= \gamma\{J\}$ and $\ms U := \Int(\gamma)\{J\}$  be finite $\gamma$-complexes, more precisely:  
\begin{flalign*}
\ms Z :& \quad \quad \bigoplus_{j_{0} \in J_0} \k_{j_0 + \gamma} \to \bigoplus_{j_{1} \in J_1} \k_{j_{1} + \gamma} \to \bigoplus_{j_{2} \in J_2} \k_{j_{2} + \gamma} \to \cdots \\
\ms U :& \quad \quad \dots  \to \bigoplus_{j_{2} \in J_2} \k_{j_{2} + \Int(\gamma)} \to \bigoplus_{j_{1} \in J_1} \k_{j_{1} + \Int(\gamma)} \to \bigoplus_{j_{0} \in J_0} \k_{j_{0} + \Int(\gamma)}.
\end{flalign*} 
Then, in $\D^b(\k_\V)$ we have \(\dR u_* \ms Z \simeq \Lambda\{u(J)\}\) and \((\dR u_* \ms U)[1-n] \simeq \Int(\Lambda)\{u(J)\}\), i.e.:
\begin{flalign*}
 \dR u_* \ms Z &: \quad \dots \to \bigoplus_{j_{1} \in J_1 } \k_{[u(j_{1}),+ \infty) } \to \bigoplus_{j_{1} \in J_1 } \k_{[u(j_{1}),+ \infty) } \to  \bigoplus_{j_{0} \in J_0} \k_{[u(j_{0}), + \infty)} \to 0, \\
\dR u_* \ms U &: \quad \dots \to \underbrace{\bigoplus_{j_{2} \in J_2} \k_{(u(j_{2}), + \infty) }}_{-n-1} \to \underbrace{\bigoplus_{j_{1} \in J_1 } \k_{(u(j_{1}),+ \infty) }}_{-n} \to  \underbrace{\bigoplus_{j_{0} \in J_0} \k_{(u(j_{0}), + \infty)}}_{-n+1} \to 0.
\end{flalign*}
Similarly, $\dR v_* \ms Z \simeq \Lambda^a\{v(J)\}$ and $(\dR v_* \ms U) [1-n] \simeq \Int(\Lambda^a)\{v(J)\}$.
\end{theorem}

The proof of this result relies  on the following two observations. The rest of the proof is a combination of classical facts from sheaf theory.

\begin{lemma}\label{lem:26} Let $f: \V \to \R$ be a continuous function, and let $\ms F$ be a sheaf in $\D^b(\k_\V)$. Then, the following isomorphism holds in $\D^b(\k_\R)$: $\dual_\R \dR f_! \ms F \simeq \dR f_* \dual_\V \ms F$.
The choice of an orientation on $\V$ yields an isomorphism $\dual'_\V \ms F \simeq (\dual_\V \ms F)[-n]$.
\end{lemma} 

\begin{proof}
The first statement is a consequence of local Verdier duality \cite[Prop. 3.1.10]{KS90}:
\[ \dual_\R \dR f_! \ms F \simeq \dR \HOM (\dR f_! \ms F, a_\R^! \k ) \simeq \dR f_* \dR \HOM ( \ms F, (a_\R f)^! \k ) \simeq \dR f_* \dual_\V \ms F.\]
The choice of an orientation yields an isomorphism  $\omega_\V \simeq \k_\V [n]$ (\cite[Remark 3.3.8]{KS90}) which provides an isomorphism $\dual_\V \ms F \simeq (\dual'_\V \ms F)[n]$.
\end{proof}
\begin{lemma} \label{lem_proj_barcode_constant} Let $u \in \Int(\gamma^{\circ})$ and $v \in \Int(\gamma^{\circ,a})$.
Then: \[\dR u_* (\k_{x + \gamma}) \simeq \k_{[u(x),+\infty)}; \quad \quad \dR u_* (\k_{x + \Int(\gamma)}) \simeq \k_{(u(x),+ \infty)}[1-n];\]
\[\dR v_* (\k_{x + \gamma}) \simeq \k_{(-\infty,v(x)]}; \quad \quad \dR v_* (\k_{x + \Int(\gamma)}) \simeq \k_{(-\infty , v(x))}[1-n].\]
\end{lemma}
\begin{proof} 
For any $x \in V$, let $\tau_x  : \V \to \V$ be given by $y \mapsto x + y$. Then, $\tau_{-x}^{-1} \k_{\gamma} \simeq \k_{x + \gamma}$ and $u \tau_x = \tau_{u(x)} u$.
It follows from \cite[Lemma 4.8]{BP23}  that:  $\dR u_* (\k_{x + \gamma}) \simeq \dR (\tau_{u(x)} u)_* (\k_{\gamma}) \simeq \tau_{-u(x)}^{-1} \k_{_{[0,+ \infty)}} \simeq \k_{[u(x),+\infty)}.$
The second statement follows from the previous one, combined with Lemmas \ref{lem_dual_calcul} and \ref{lem:26}. \end{proof}

\begin{proof}[Proof of Theorem~\ref{thm_proj_bar}]
First we prove the statement for the closed conic-complex, then we deduce the statement for the open conic-complex.

It follows from Lemma \ref{lem_proj_barcode_constant} that $\k_{a + \gamma}$ are $u_*$- and $v_*$-acyclic. Thus, the result is a straightforward consequence of Lemma \ref{lem_acyclic} and Lemma \ref{lem_proj_barcode_constant}.

Only the case of $u$ will be treated here, the proof holds mutatis-mutandis for $v$. It follows from Proposition \ref{prop_constr} and from \cite[Proposition 3.4.3]{KS90} that $\ms U$ and $\ms Z$ are cohomologically constructible, in particular  $\dual_\V \dual_\V \ms U \simeq \ms U$ and $\dual_\V \dual_\V \ms Z \simeq \ms Z$. The following duality argument allows us to restrict the focus to the case of a finite closed conic-complex:
\[\dR u_* \ms U \simeq \dR u_* \dual_\V \dual_\V \ms U \simeq \dual_\R( \dR u_* \dual_\V \ms U) \simeq \dual_\R'(\dR u_* \dual'_\V \ms U)[1-n].\]
Indeed, Proposition \ref{prop_dual_resolution} ensures that $\dual_\V' \ms U \simeq \gamma\{J\}$.
Thus, the preceding part of this proof yields the following expression:
\[\dR u_* (\dual'_\V \ms U) : \quad \dots \to \bigoplus_{j_{1} \in J_1 } \k_{[u(j_{1}),+ \infty) } \to  \bigoplus_{j_{0} \in J_0} \k_{[u(j_{0}), + \infty)}.\]
A second application of Proposition \ref{prop_dual_resolution} concludes the proof.
\end{proof}

\section{From 1-d conic complexes to filtered cochain complexes}

\subsection{From finite conic-complexes to filtered finite cochain complexes}\label{subsec:qsw_conic}

Let again $\Lambda:=[0, +\infty)$. The barcode of a finite open $\Int(\Lambda)$-complex can be obtained from that of a finite filtered cochain complex as follows. This reduction is an essential step in our approach for the computation of the projected barcode of a conic-complex.

We start by recalling the notion of finite filtration and finite filtered cochain complex.
\begin{definition}
\begin{enumerate}[(i)]
\item A  finite filtration $C(\bullet)$ in $\Ch^b(\k)$ is the data of
a finite sequence $(C(i))_{0 \leq i \leq m}$ and $(f_{i-1})_{1 \leq i \leq m}$ where, $C(0)=0$, $C(i)$ is an object of $\Ch^b(\k)$ and $f_{i-1} \colon C(i-1) \to C(i)$ is a monomorphism for $1 \leq i \leq m$ in $\Ch^b(\k)$. This is summarized by
\begin{equation*}
\begin{tikzcd} 0 = C(0) \ar[r, "f_0", rightarrowtail] & \cdots \ar[r, "f_{m-1}", rightarrowtail] & C(m)\end{tikzcd}.
\end{equation*}
\item A filtered finite cochain complexes is the data of a cochain complex $C \in \Ch^b(\k)$ together with a finite filtration $(C(i))_ {i=0}^m$, $(f_{i-1})_{1 \leq i \leq m}$ such that $C(m)=C$.
\end{enumerate}
\end{definition}
We will identify a finite filtered complex with its filtration if there is no risk of confusion.
%
%
%
Such a filtration can be seen in two ways:
\begin{itemize}
    \item as a functor, i.e., an object of $\Fun([m],\Ch^b(\k))$;
    \item as a complex of persistence modules, i.e., an object of the abelian category $\Ch^b(\Fun([m],\Mod(\k)))$.
\end{itemize}
Both viewpoints will be used in the following.

A filtered cochain complex composed of finite-dimensional vector spaces yields, for any integer $d \in \Z$, a pfd persistence module $\H^d(C(\bullet))$, which by~\cite[Theorem 1.1]{C15}  admits a barcode $\B(\H^d(C(\bullet)))$, denoted by ${\B}^d(C(\bullet))$ from now on for simplicity, while ${\B}(C(\bullet))$ is a shorthand for the graded barcode $\{ {\B}^d(C(\bullet)) \}_{d \in \Z}$.

Let $\ms L = \Int(\Lambda)\{J\}$ be a finite $\Int(\Lambda)$-complex, and let $m := |J|$ be the cardinal of $J$ as a set (i.e. without counting the multiplicity of the elements). In other words:
\begin{equation*}
\ms L \quad \ldots \to 0 \to \bigoplus_{j \in J_0 } \k_{(j, + \infty) } \to \cdots \to \bigoplus_{j \in J_\nu} \k_{(j, + \infty)} \to 0 \to \cdots .
\end{equation*}
The sheaf $\ms L$ is a $\gamma$-sheaf for the cone $\gamma=\Lambda$. Consider the complex $\ms L^\a$ of Alexandrov sheaves on $\V_\mathfrak{a}$ given by
\begin{equation*}
\ms L^\mathfrak{a} \quad \ldots \to 0 \to \bigoplus_{j \in J_0 } \k_{[j, + \infty) } \to \cdots \to \bigoplus_{j \in J_\nu} \k_{[j, + \infty)} \to 0 \to \cdots
\end{equation*}
This complex is such that $\alpha^{-1} \ms L^\mathfrak{a} \simeq \ms L$ and corresponds through the morphism $\theta_*$ to the complex of free persistence modules on $(\R,\leq)$:
\begin{equation*}
L \quad \ldots \to 0 \to \bigoplus_{j \in J_0 } \k^{[j, + \infty) } \to \cdots \to \bigoplus_{j \in J_\nu} \k^{[j, + \infty)} \to 0 \to \cdots
\end{equation*}
We now associate a finite filtered cochain complex to $L$. The multiset $J$ of generators of $\ms L$ is a subset of $\R$. Thus, we order the generators using the increasing order, and we denote by $i(r)$  the $r^{th}$ generator in this order (numbering from one) and set arbitrarily $i(0)=i(1)-1$. Thus, we have a morphism of poset $i \colon [m] \to \R$, $ r \mapsto i(r)$. Let $i^{-1}L:=L \circ i$.
\begin{proposition}\label{prop_finite_filtration}
The filtration $i^{-1}L$ is a finite filtration.  
\end{proposition}

\begin{proof}
Consider $i_!(i^{-1}L)$ the left Kan extension along $i$ of $i^{-1}L$. Since $[m]$ is a finite poset, the functor $i_!$ is exact and a straightforward computation shows that $i_!(i^{-1}L) \simeq L$. In particular, $(i^{-1}L)(r)\simeq L(i(r))$ and for every $r \in [m]$ the structural morphism $i^{-1}L(r-1) \to i^{-1}L(r)$ is a monomorphism as it is induced by the internal morphisms of indicator modules.
\end{proof}

%
%
%
\begin{lemma} \label{lem:calcul_inter}
    Given  an interval  $I=\lb r,s \rb$ of $[m]$ and its associated indicator module $\k^I$, if $s \neq m$ then $i_! \k^I \simeq \k^{[i(r),i(s+1))}$, else $i_! \k^I \simeq \k^{[i(r),+\infty)}$.
\end{lemma}
\begin{proof}
This follows from $i_!\k^I(t) \simeq \underset{\{i(h) \leq t\}}{\colim} \k^I(h) \simeq \k^I(\max\{h \mid i(h) \leq t\})$.
\end{proof}
Let $\mc I_{[m]}$ (resp. $\mc I_{\R}$) be the set of intervals of $[m]$ (resp. $\R$) and consider $i_! \colon \mc I_{[m]} \to \mc I_{\R}$~defined~by
\begin{align*}
i_!(\lb r, s \rb)= 
\begin{cases}
[i(r),i(s)) \quad \textnormal{if} \; s \neq m \\
[i(r), +\infty) \quad \textnormal{if} \; s = m
    \end{cases}
\end{align*}

For $I_{\mathfrak{a}}=\Omega_\mathfrak{a} \cap F_\mathfrak{a} $ a locally closed interval of $\R_\mathfrak{a}$, we define $\alpha^{-1}(I_\mathfrak{a}):= \Int(\Omega_\mathfrak{a}) \cap \overline{F_\mathfrak{a}}$. 

\begin{theorem}\label{thm_red_qsw}
For every integer $k$, the application 
$\alpha^{-1}i_!: \B^k(i^{-1}L) \to  \B^k(\ms L)$, $I \mapsto \alpha^{-1} i_!(I)$
is a well-defined bijection of multisets.
\end{theorem}
\begin{proof}
We first prove that $\alpha^{-1}i_!: \B^k(i^{-1}L) \to  \B^k(\ms L)$ is a well defined bijection. Remark that $\alpha^{-1} \theta^{-1} i_! i^{-1}L \simeq \alpha^{-1} \ms L^\mathfrak{a} \simeq \ms L$ and that the functors $\alpha^{-1}, \; \theta^{-1}, \; i_!, \; i^{-1}$ are exact. Hence, $\alpha^{-1} \theta^{-1} i_! \H^k(i^{-1}L) \simeq \H^k(\ms L)$. Moreover,  $\H^k(i^{-1}L)$ admits a barcode decomposition $\bigoplus_{I \in \B^k(i^{-1}L)} \k_{I}$. Hence, using  Lemma \ref{lem:calcul_inter} and Proposition \ref{prop_pullback_alpha}, we get
\begin{align*}
\alpha^{-1} \theta^{-1} i_! \H^k(i^{-1}L) &\simeq  \bigoplus_{I \in \B^k(i^{-1}L)} \k_{\alpha^{-1} i_!(I)}.
\end{align*}
Furthermore, $\H^k(\ms L)$ admits a barcode decomposition $\bigoplus_{I^\prime \in \B^k( \ms L)} \k_{I^\prime}$. Thus, 
\begin{equation*}
\bigoplus_{I^\prime \in \B^k( \ms L)} \k_{I^\prime} \simeq \bigoplus_{I \in \B^k(i^{-1}L)} \k_{\alpha^{-1} i_!(I)}. 
\end{equation*}
The unicity of the barcode decomposition (Theorem \ref{thm:Decomposition}) yields that the application $\alpha^{-1}i_! : \B^k(i^{-1}L) \to  \B^k(\ms L) $ is a well-defined bijection. 
\end{proof}


\section{Extension of the persistence algorithm} \label{sec:extension_algo}

The persistence algorithm 
applies to cochain complex filtrations arising from simplex-wise filtrations of simplicial complexes. However, as we shall see in Section~\ref{sec:conic-complexes}, conic-complexes yield cochain complex filtrations that do not necessarily fit this framework. Fortunately, the only differences are the following ones:
(1) several "simplices" can appear at the same time in the filtration;
(2) boundary matrices do not admit a fixed number of non zero entries depending on the dimension of the simplex.
These differences can be overcome by a straightforward extension of the persistence algorithm.

\subsection{Simplex-wise filtered cochain complexes}

%
%
\begin{definition}\label{def:simplex-wise}
The filtration $C(\bullet)$ is called \emph{simplex-wise} if, for every $1 \leq i \leq m$, we have $\coker(f_{i-1}) \simeq \k[q_i]$  for some $q_i\in\Z$. In this case, each morphism~$f_i$ in the sequence is called a {\em simplex insertion}.
\end{definition}

\begin{remark}
The terminology comes from simplex-wise filtered simplicial complexes, which naturally yield simplex-wise filtrations.
\end{remark}

From now on, we assume $C(\bullet)$ to be simplex-wise. For any $0 \leq i \leq m$ and any integer $\ell$, we denote by $\beta^\ell(C(i))$ the dimension of $\H^\ell(C(i))$.

\begin{proposition}\label{prop_creator_destroyer_complex}
For any $1 \leq i \leq m$, we have $\beta^\ell(C(i)) = \beta^\ell(C(i-1))$ for all $\ell \notin \{q_i,\, q_i+1\}$, and we have either $\beta^{q_i}(C(i)) = \beta^{q_i}(C({i-1})) + 1$ (index $i$ is then called a \emph{creator} of degree $q_i$) or $\beta^{q_i +  1}(C(i)) = \beta^{q_i + 1}(C({i-1})) - 1$ (index $i$ is then called a \emph{destructor} of degree $q_i$). 
\end{proposition}
\begin{proof}
Since $f_{i-1}$ is a morphism of complexes, we have the following commutative diagram, where the internal direct-sum decomposition $C(i)^{q_i}=\im(f_{i-1}^{q_i})\oplus\k_{q_i}$ involves an arbitrary vector space complement  $\k_{q_i}\simeq\k$ of~$\im(f_{i-1}^{q_i})$ in~$C(i)^{q_i}$:
\[\begin{tikzcd}
	\cdots & C(i-1)^{q_i - 1} & C(i-1)^{q_i} & C(i-1)^{q_i + 1} & \cdots \\[10pt]
	\cdots & C(i)^{q_i - 1} & \im(f_{i-1}^{q_i}) \oplus \k_{q_i} & C(i)^{q_i + 1} & \cdots 
	\arrow[from=1-1, to=1-2]
	\arrow["d_{i-1}^{q_i -1}", from=1-2, to=1-3]
	\arrow["d_{i-1}^{q_i}", from=1-3, to=1-4]
	\arrow["d_{i-1}^{q_i+1}", from=1-4, to=1-5]
	\arrow["d_{i}^{q_i -1}", from=2-2, to=2-3]
	\arrow["d_{i}^{q_i}", from=2-3, to=2-4]
    \arrow["d_{i}^{q_i+1}", from=2-4, to=2-5]
	\arrow[from=2-1, to=2-2]
	\arrow["f_{i-1}^{q_i-1}", "\isoarrow"', from=1-2, to=2-2]
	\arrow["f_{i-1}^{q_i}", tail, from=1-3, to=2-3]
	\arrow["f_{i-1}^{q_i+1}", "\isoarrow"', from=1-4, to=2-4]
\end{tikzcd}\]
Then, for all $\ell\neq q_i$ we have:
\[\begin{array}{ccccccc}
    \ker (d_{i-1}^\ell) &=& \ker (f_{i-1}^{\ell+1}\circ d_{i-1}^\ell) &=& \ker (d_i^\ell\circ f_{i-1}^\ell) &\simeq& \ker (d_i^\ell);\\[1ex]
\im (d_{i-1}^\ell) &\simeq& \im (f_{i-1}^{\ell+1}\circ d_{i-1}^\ell) &=& \im (d_i^\ell \circ f_{i-1}^\ell) &=& \im (d_{i}^\ell).
\end{array}\]

Letting $\eta_{i} :=  \dim \ker (d_i^{q_i}) - \dim \ker(d_{i-1}^{q_i})$, we deduce that:
\begin{itemize}
\item $\beta^\ell(C(i)) = \beta^\ell(C(i-1))$ for all $\ell\notin\{q_i, q_i+1\}$;
\item $\beta^{q_i} (C(i)) = \beta^{q_i}(C(i-1)) + \eta_{i}$;
\item $\beta^{q_i+1}(C(i)) = \beta^{q_i + 1}(C(i-1)) + \eta_{i} - 1$ (by the rank-nullity theorem).
\end{itemize}

\noindent Meanwhile, the commutative diagram also implies that:
\[\ker(d_{i-1}^{q_i}) = \ker (f_{i-1}^{q_i+1}\circ d_{i-1}^{q_i}) = \ker (d_i^{q_i}\circ f_{i-1}^{q_i}) \simeq \ker(d_i^{q_i})\cap \im(f_{i-1}^{q_i}),\]
where $\im(f_{i-1}^{q_i})$ has codimension~$1$ in the domain of~$d_i^{q_i}$. It follows that $\eta_{i} \in \{0,1\}$, which concludes the proof.
\end{proof}

\begin{remark}
When $C(\bullet)$ is the finite filtration with a simplex-wise filtered simplicial complex, the notions of creator and destructor coincide with the usual notions of creator simplex and destructor simplex.
\end{remark}

Proposition~\ref{prop_creator_destroyer_complex} implies that, for each time~$i\in [m]\setminus\{0\}$, exactly one interval endpoint of~${\B}(C(\bullet))$ is located at~$i$. As a consequence, the graded barcode ${\B}(C(\bullet))$ pairs each destructor~$j$ in degree~$q_j$ with a creator~$i$ in degree~$q_i=q_j+1$.

\begin{definition}
 The pairs $(i,j)$ thus created are called \emph{persistence pairs}. Creators that remain unpaired are referred to as \emph{essential} creators.
\end{definition}

\subsection{From arbitrary finite filtrations to simplex-wise filtrations}\label{sec_qsimplexwise}

A reduction to the simplex-wise case allows for the application of the persistence algorithm in the more general setting where several simplex insertions may happen at the same time.

As before, let $C(\bullet)$ be a finite filtration in $\Ch^b(\k)$, that is: 
\begin{equation*}
\begin{tikzcd} 0 = C(0) \ar[r, "f_0", rightarrowtail] & \cdots \ar[r, "f_{m-1}", rightarrowtail] & C(m)\end{tikzcd}.
\end{equation*}
We have the following commutative diagram for every $1\leq i\leq m$, where each internal direct-sum decomposition $C(i)^j=\im(f_{i-1}^j) \oplus N_j$ involves an arbitrarily-chosen vector-space complement  $N_{j}$ of~$\im(f_{i-1}^j)$ in~$C(i)^{j}$---note that $d_i^j$ maps $\im(f_{i-1}^j)$ to $\im(f_{i-1}^{j+1})$ but may not always map $N_j$ to $N_{j+1}$:

\begin{equation*}
\adjustbox{scale=0.999,center}{%
\begin{tikzcd}
	\cdots & C(i-1)^{j-1} & C(i-1)^{j} & C(i-1)^{j+1} & \cdots \\[10pt]
	\cdots & \im(f_{i-1}^{j-1}) \oplus N_{j-1} & \im(f_{i-1}^{j}) \oplus N_{j} & \im(f_{i-1}^{j+1}) \oplus N_{j+1} & \cdots 
	\arrow[from=1-1, to=1-2]
	\arrow["d_{i-1}^{j-1}", from=1-2, to=1-3]
	\arrow["d_{i-1}^{j}", from=1-3, to=1-4]
	\arrow["d_{i-1}^{j+1}", from=1-4, to=1-5]
	\arrow["d_{i}^{j-1}", from=2-2, to=2-3]
	\arrow["d_{i}^{j}", from=2-3, to=2-4]
    \arrow["d_{i}^{j+1}", from=2-4, to=2-5]
	\arrow[from=2-1, to=2-2]
	\arrow["f_{i-1}^{j-1}", tail, from=1-2, to=2-2]
	\arrow["f_{i-1}^{j}", tail, from=1-3, to=2-3]
	\arrow["f_{i-1}^{j+1}", tail, from=1-4, to=2-4]
\end{tikzcd}
}
\end{equation*}

In order to turn~$C(\bullet)$ into a simplex-wise filtration, our approach is to write every morphism $f_{i-1}$ in the sequence as  a composition of finitely many morphisms that are each a simplex insertion as per Definition~\ref{def:simplex-wise}. For this, we start from the complex $C(i-1)$, map it to the image of~$f_{i-1}$, then insert the vector space complements $N_j$ one by one,  until we eventually get $C(i)$. The key point is to insert the vector space complements $N_j$ in decreasing order of~$j$, so that the codomain of~$d_i^j$ is ready when $N_j$ is being added to $\im (f_{i-1}^j)$. The details of the procedure are given in Appendix~\ref{sec:appendix_proc}.

Once this is done, we have decomposed $f_{i-1}$ into a finite sequence of simplex insertions. We can index the intermediate complexes over an arbitrary finite subset $T_i$ of the open interval~$(i-1, i)$. Doing this for every $1\leq i\leq m$ yields a simplex-wise filtration~$[C(\bullet)]_{\mathrm{sw}}$, indexed over $[m]\cup \bigcup_{i=1}^m T_i$, whose restriction to the  grid~$[m]$ is isomorphic to $C(\bullet)$ by construction.  Since restrictions preserve finite direct sums, we conclude:

\begin{proposition}\label{prop_finite_simplexwise_barcode}
For each degree~$j$, the map $[a,\, b) \mapsto [\lceil a \rceil,\, \lceil b \rceil)$
induces a bijection of multisets $\B' \to \B^j(C(\bullet))$, where 
$\B' := \{[a,b) \in \B^j([C(\bullet)]_{\mathrm{sw}}) \mid \lceil a \rceil < \lceil b \rceil \}$.
\end{proposition}

\subsection{Persistence algorithm on simplex-wise filtered cochain complexes} \label{sec:twistoncochaincomplex}

The persistence algorithm recovers the persistence pairs from the boundary matrix associated to a filtered simplicial complex. The differentials in a simplex-wise filtered cochain complex yield such a matrix. 
%

\begin{definition}
The \emph{boundary matrix} of a simplex-wise filtration $C(\bullet)$ is the matrix $\partial \in \k^{m \times m}$ defined by $\partial_{i,j} := \alpha_{ij}$, where the map $\k_j \to \k_i$ arising from $d_j^{q_j}$ is the scalar multiplication by~$\alpha_{ij}$.
\end{definition}

\begin{remark} When $C(\bullet)$ is the cochain complex of a simplex-wise filtered simplicial complex~$K$, its boundary matrix coincides with the usual co-boundary matrix of~$K$.  \end{remark}

We can now feed our boundary matrix~$\partial$ as input to the persistence algorithm, which then decomposes it as $\partial = RU$ with $R$ reduced and $U$ upper triangular. The proofs of the Pairing Uniqueness Lemma and of the fact that $(i,j)$ is a persistence pair if, and only if, $\low(R_j) = i$, hold verbatim in our context---see~\cite[Proposition 3.5 and Theorem 3.6]{DW22} for the case~$\k=\Z_2$. 
%
%
Consequently:
\begin{corollary}
Given $C(\bullet)$ simplex-wise, for any $d \in \Z$ the $d$-th barcode is given by:
\[ \B^d(C(\bullet)) = \{ [a,b) \mid (a,b) \in \mc P_d \} \coprod \{ [a, +\infty) \mid a \in \mc E_d \}, \]
where $\mc P_d$ and $\mc E_d$ denote respectively: the multiset of persistence pairs with a creator of degree $d$, and the multiset of essential creators of degree $d$.
\end{corollary}




%
%

\section{Algorithm for projected barcode} \label{sec:algo_proj_barcode}

Let us summarize the situation. The computation of the projected barcode of a finite open conic-complex on $\R^n$ along a linear form $u \in \Int(\gamma^\circ)$ starts by reducing the problem to computing the barcode of a finite $\Int(\Lambda)$-complex on $\R$ (see Section \ref{subsec:direct_image} and Theorem \ref{thm_proj_bar}). This finite $\Int(\Lambda)$-complex on $\R$ admits a corresponding complex $L$ of free persistence modules that shares the same observable barcode (Section \ref{subsec:qsw_conic} and Proposition \ref{prop_pullback_alpha}). According to Section \ref{subsec:qsw_conic}, one can associate to $L$ a finite filtered cochain complex $i^{-1}L$ such that the barcode of $L$ can be reconstructed from the barcode of $i^{-1}L$. Meanwhile, according to Section~\ref{sec_qsimplexwise},  one can associate to a finite filtered cochain complex, here $i^{-1}L$, a simplexwise filtered cochain complex $[i^{-1}L]_{\mathrm{sw}}$ whose barcode gives the barcode of $i^{-1}L$ via restriction. Finally, the barcode of  $[i^{-1}L]_{\mathrm{sw}}$ can be computed using the standard persistence algorithm. In this section, we combine all these ingredients together into an effective algorithm, focusing on the special case where the cone $\gamma$ is the positive orthant because this is the standard setup in topological data analysis and because the details of the algorithm are simpler in this case---especially in the 2-parameter setting, as we shall see in Section~\ref{sec:algo_2-params}.

\bigskip

Let $\ms L := \Int(\gamma)\{G\}$ with $G = \{ g_i \}_{i=1}^\mu$ a finite multi-subset of $\V = \R^n$ with $n > 1$ endowed with the norm~$\norm{\cdot}_{\infty}$. We let $m:=\vert G \vert$ be the cardinal of $G$ as a set and $\mu$ its cardinal as a multiset (i.e. taking into account the multiplicities of the elements of $G$). The dual vector space $\V^*$ endowed with the operator norm identifies with $(\R^n, \norm{\cdot}_1)$. Let our cone $\gamma$ be the positive orthant $[0,+\infty)^n$, which corresponds to $\leq$ being the product partial order on $\R^n$. Then, $\gamma^\circ$ identifies also with $[0,+\infty)^n$. A linear form is called \emph{relevant} if it is of unit length and belongs to $\Int(\gamma^\circ)$.

The differentials of the complex $\dR u_* \ms L$ for any relevant linear form $u \in \V^*$ admit the same matrices as the differentials of $\ms L$ (this follows from the definition of the pushforward for sheaves). In particular, according to Section~\ref{subsec:qsw_conic},  the projected barcode along different linear forms essentially differs when the preorders on the generators of~$\ms L$ induced by these linear forms differ. This symmetry can be translated into a combinatorial structure that encodes the essential information describing the projected barcode, and that takes the form of a hyperplane arrangement augmented with one barcode template per face. We call this combinatorial structure the \emph{projected barcode template} of~$\ms L$, or $\PBT(\ms L)$ for short. 
As in RIVET~\cite{LW15}, its construction can be done in a pre-processing step, and a point location data structure built on top of the arrangement allows then for the efficient query of projected barcodes along linear forms.

\subsection{Arrangement of hyperplanes and point location}
\label{sec:arrangement}

We describe the hyperplane arrangement mentioned above. Its use will be clarified in the next section. The structure of the hyperplane arrangement is controlled on the one hand by the geometry of the domain of the $\gamma$-linear projected barcode, that is: the linear forms under consideration are elements of $\Int(\gamma^{\circ}) \cap \S^*$, where $\S^*$ can be identified with the unit sphere in the norm $\norm{\cdot}_1$;  on the other hand, the structure of the hyperplane arrangement is controlled by the set of generators of the conic-complex $\ms L$.

Here are the details. Observe that the set of relevant linear forms is the relatively open simplex $S=\{ v \in \Int(\gamma^\circ) \mid \norm{v}_1 = 1 \}$ in $\V^*$, that is: the intersection of $\Int(\gamma^\circ)$ with the affine hyperplane $T$ of equation $\sum_{i=1}^n x_i = 1$. Let $H_0$ be the collection of vector hyperplanes $H_{ij}\subset \V^*$ of linear forms vanishing on $g_{ij}:=  g_i - g_j$, for $i<j \in \lb 1, m \rb$, and let $H$ be the collection of their restrictions to~$T$---which are affine hyperplanes of~$T$. Without loss of generality, we assume that all the elements of~$H_0$ are distinct. Locating a relevant linear form~$u$ in the associated arrangement~$\mc A(H_0)$ is then equivalent to locating~$u$ in the restriction of~$\mc A(H_0)$ to~$T$, which itself is the arrangement~$\mc A(H)$  associated to~$H$ in~$T$. Henceforth, we implicitly identify $T$ with $\R^{n-1}$ via the projection $(x_1, \ldots, x_n) \mapsto (x_2, \ldots, x_n)$, which identifies $\mc A(H)$ with an arrangement of affine hyperplanes in $\R^{n-1}$.
Here, we use the notion of $k$-face as in \cite{M93} with the difference that here a \emph{face} refers to a face of maximal (affine) dimension whenever the dimension is not specified. Moreover, we call {\em relevant} any face of $\mc A(H)$ that intersects $\Int(\gamma^\circ)$. 

Note that the hyperplanes in $\mc A(H)$ may very much be in degenerate position. Therefore, we do not use Meiser's original data structure~\cite{M93} for point location in the arrangement, but rather the variant described in~\cite{ezra2020decomposing}, which can handle degeneracies and has the finer and more accurate complexity bounds of $\mr O(\frac{(n-1)^5}{\log (n-1)}\log m)$ query time, and $\mr O(m^{(6+o_n(1))(n-1)})$ storage space and maximum expected pre-processing time, for our $\mr O(m^2)$ hyperplanes in $n-1$~dimensions, where $o_n(1)$ denotes a term depending only on~$d$ that goes to zero as~$d\to\infty$.

\subsection{Projected barcode template}
\label{sec:proj_barcode_template}

A relevant face $\sigma$ of $\mc A(H)$ induces a well-defined total order $\preceq_\sigma$ on $G$ such that $g \preceq_\sigma g'$ if, and only if, $u(g) \leq u(g')$ for every $u \in \sigma$. Without loss of generality we assume that $\{ g_i \}_{i=1}^m$ is an $\preceq_\sigma$-increasing enumeration of $G$. Let $u,v \in \sigma$ be two relevant linear forms, and let $\ms L_u := \dR u_* \ms L$ and $\ms L_v := \dR v_* \ms L$. According to Section~\ref{subsec:qsw_conic}, one can construct two corresponding finite filtered cochain-complexes $i_u^{-1} L_u$ and $i_v^{-1} L_v$ indexed over $[m]$. Since the two linear forms $u$ and $v$ induce the same order $\preceq_\sigma$, the two filtered complexes $i_u^{-1} L_u$ and $i_v^{-1} L_v$ are isomorphic, therefore they have the same barcode, called the {\em barcode template} associated to~$\sigma$ and denoted by $\B_\sigma(\ms L)$. By Theorem \ref{thm_red_qsw}, this barcode template can be used to retrieve the projected barcode along any linear form in $\sigma$, a result that we rephrase below using our current notation: 

\begin{corollary}\label{cor:PBT_to_PB}
Let $u$ be a relevant linear form in a relevant face $\sigma$ of $\mc A(H)$, and let $k$ be an integer. The application $\pi_u : \B_\sigma^k(\ms L) \to \B^k(\dR u_* \ms L)$ given below is a bijection of multisets:
\[ \pi_u ([i,j)) := \begin{cases} (u(g_i),u(g_j)] \text{ if } j \neq m, \\ (u(g_i), + \infty ) \text{ if } j = m. \end{cases}\]

\end{corollary}

The arrangement $\mc A(H)$, together with its point location data structure, and augmented with the collection of barcode templates $\B_\sigma(\ms L)$ associated to all the relevant faces $\sigma$ in the arrangement, forms what we call the \emph{projected barcode template} of $\ms L$, or $\PBT( \ms L)$ for short.

\bigskip

As in RIVET, once the arrangement~$\mc A(H)$ has been computed, the barcode templates $\overline{\B}_\sigma(\ms L)$ for all the relevant faces $\sigma\in \mc A(H)$ are computed via a walk through the dual graph of the arrangement---instead of a dfs or bfs, in order to maintain only a single boundary matrix. Following the approach of~\cite{LW15}, to each edge $[\sigma, \tau]$ in the dual graph we assign a cost that is the minimum number of elementary transpositions required to transition from $\preceq_\sigma$ to $\preceq_\tau$ (this number is given, e.g., by insertion sort). Then, we build a path in the graph that visits every vertex (possibly multiple times) and whose total cost is at most twice that of the optimal such path. To build our path we use the standard $2$-approximation scheme for the traveling salesman's problem, applying first Kruskal's minimum spanning tree algorithm to the graph, then Hierholzer's Eulerian cycle algorithm to the tree with doubled edges. Once our path is built, we walk along it, computing the barcode template of the first face from scratch using the standard boundary matrix reduction algorithm, then updating the barcode template by further reducing the boundary matrix using the vineyards algorithm~\cite{CEM06} at each transition between adjacent faces. As the number of elementary transpositions involved at each transition is in $\mr O(m^2)$, the algorithmic cost of each transition is in $\mr O(m^3)$: while this is no better than recomputing the barcode template from scratch at each faces in the worst case, in practice the number of elementary transpositions involved can be much smaller than $m^2$.

Overall, since the number of hyperplanes in~$H$ is in $\mr O(m^2)$, the size of the dual graph of the arrangement is in $\mr O(m^{2n-2})$, so the number of edges in our path is also in $\mr O(m^{2n-2})$, and so the total worst-case cost of computing the path then the barcode templates for all the relevant faces is in $\mr O(m^{2n-2} \mu^3)$. This cost is dominated by that of the construction of the arrangement and its point location data structure in Section~\ref{sec:arrangement}.

\subsection{Queries}\label{subsec:queries}

Given a relevant linear form $u$ whose $\gamma$-linear projected barcode we want to compute, we use the point location data structure from Section~\ref{sec:arrangement} to retrieve the face of the arrangement~$\mc A(H)$ that contains~$u$ in $\mr O (\frac{(n-1)^5}{\log (n-1)} \log m)$ time. Then, two distinct cases can arise :
\begin{enumerate}
    \item the face containing $u$ denoted~$\sigma$ is a (necessarily relevant) face;
    \item the face containing $u$ denoted~$\tau$ has not maximal (affine) dimension.
\end{enumerate}
In case~1, we simply apply Corollary~\ref{cor:PBT_to_PB} to get the projected barcode of~$u$ from the barcode template associated to~$\sigma$ in $\mr O(\mu n)$ time as $\mu$ provides an upper bound on the number of bars and $n$ is the number of coordinates of each generator that $\pi_u$ maps through~$u$. In case~2, we do roughly the same thing, except that now $u$ lies on the boundary of some relevant face~$\sigma$ incident to~$\tau$, so Corollary~\ref{cor:PBT_to_PB} must be replaced by the following barcode continuity result, which is a direct consequence of the fact that $\tau$ induces a pre-order on~$[m]$ that is obtained by quotienting the order induced by~$\sigma$ by the fibers of~$u$ : 
\begin{corollary}\label{cor:continuity_barcode}
Let $u$ be a relevant linear form belonging to a face~$\tau$ of~$\mc A(H)$ that is not of maximal (affine) dimension, and let $\sigma$ be an incident relevant face. For any integer $k$, the application $\pi_u : \{ [i,j) \in \B_\tau^k(\ms L) \mid u(g_i) \neq u(g_j) \} \to \B^k(\dR u_* \ms L)$ given below is a bijection of multisets:
\[ \pi_u ([i,j)) := \begin{cases} (u(g_i),u(g_j)] \text{ if } j \neq m \\ (u(g_i), + \infty ) \text{ if } j = m \end{cases} \]
\end{corollary}
\noindent Overall, the $\gamma$-linear projected barcode of~$u$ is computed in $\mr O(\mu n + \frac{(n-1)^5}{\log (n-1)} \log m)$ time.
%
%
%
\subsection{The special case of 2-parameter persistence modules} \label{sec:algo_2-params}
%
%

In the setting of 2-parameter persistence, our arrangement~$\mc A(H)$ is 1-dimensional and has a particular structure, which simplifies the approach significantly. 
Observe that $H_{ij}$ intersects $S$ if, and only if, $g_{ij} := (a_{ij},b_{ij}) \notin \Int(\gamma) \cup \Int(\gamma^a)$, i.e.,  $a_{ij}b_{ij}<0$. In this case, we have $H = \{ (\frac{|b_{ij}|}{|a_{ij}|+|b_{ij}|},\frac{|a_{ij}|}{|a_{ij}|+|b_{ij}|}) \in \R^2 \mid a_{ij}b_{ij} < 0 \}$, and via the identification of $T$ with $\R$, we get that $H = \{ \frac{|a_{ij}|}{|a_{ij}|+|b_{ij}|} \in \R \mid a_{ij}b_{ij} < 0 \}$. The point location of a given query linear form can then be solved in $\mr O(\log m)$ time by a simple binary search.
In pre-processing, we compute the $O(m^2)$ hyperplanes~$H_{ij}$ and the boundary points of the faces in the induced 1-dimensional arrangement~$\mc A(H)$ in $O(m^2)$~time. Then, we compute the barcode templates as in Section~\ref{sec:proj_barcode_template} in $O(m^2 \mu^3)$ time. This worst-case bound is comparable to that of RIVET, however the computation and storage of our arrangement are much simpler. 
\section{Example}\label{sec:example}

In this section we provide a running example of the procedure developed in the present paper. The goal is here to compute of the projected barcode template of the following persistence module, denoted by $M$.

\begin{figure}[ht]
    \centering
    \includegraphics[scale = 1]{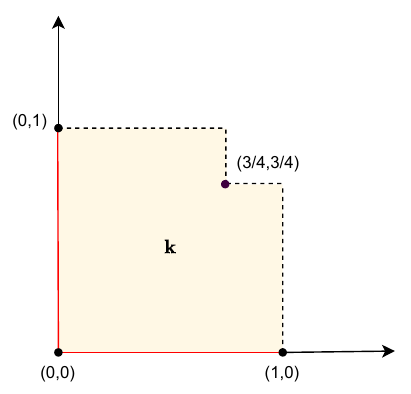}
    \caption{The persistence module $M$}
\end{figure}
More precisely, we set
\begin{align*}
g_0=(0,0) && g_1=(1,0) && g_2=(0,1)\\ 
g_3=(\frac{3}{4},\frac{3}{4}) && g_4=(\frac{3}{4},1) && g_5=(1,\frac{3}{4}). 
\end{align*}
and define the module $M$ as follows :
\[ M(x,y) := 
\begin{cases}
\k \text{ if } 0 \leq x \leq 3/4 \text{ and } 0 \leq y \leq 1,\\
\k \text{ if } 0 \leq x \leq 1 \text{ and } 0 \leq y \leq 3/4,\\
0  \text{ otherwise}.
\end{cases}
\]
with identity as morphisms between non $0$ vector spaces. Moreover, $M$ admits the following free resolution:
\begin{align*}
   \k^{\upp(g_4)} \oplus \k^{\upp(g_5)} \stackrel{d_M^{-2}}{\longrightarrow} \k^{\upp(g_1)} \oplus \k^{\upp(g_2)} \oplus \k^{\upp(g_3)} \stackrel{d_M^{-1}}{\longrightarrow} \k^{\upp(g_0)} 
\end{align*}
 with $d_M^{-2}(s_4,s_5)=(s_4,-s_4-s_5,s_5)$ and $d_M^{-1}(s_1,s_2,s_3)=s_1+s_2+s_3$.
The conic complex $\Int(\gamma)\{J\}$ associated with the above free resolution is :
\begin{align*}
   \ms U \colon \k_{g_4+\Int(\gamma)} \oplus \k_{g_5+\Int(\gamma)} \stackrel{d_{\ms{U}}^{-2}}{\longrightarrow} \k_{g_1+\Int(\gamma)} \oplus \k_{g_2+\Int(\gamma)} \oplus \k_{g_3+\Int(\gamma)} \stackrel{d_{\ms{U}}^{-1}}{\longrightarrow} \k_{g_0+\Int(\gamma)}
\end{align*}
 with $d_{\ms{U}}^{-2}(s_4,s_5)=(s_4,-s_4-s_5,s_5)$ and $d_{\ms{U}}^{-1}(s_1,s_2,s_3)=s_1+s_2+s_3$ and the multiset of generators $J=\{g_0,g_1,g_2,g_3,g_4,g_5\}$. Note that the non-zero terms of this complex are in cohomological degrees $-2,-1$ and $0$.

We now compute the $g_{ij}=g_i-g_j$ with $i<j$. We obtain
\begin{equation*}
\begin{matrix}
g_{01}=(-1,0)                      &g_{12}=(1,-1)                    & g_{23}=(-\frac{3}{4},\frac{1}{4})& g_{34}=(0,-\frac{1}{4})& g_{45}=(-\frac{1}{4},\frac{1}{4})\\[3ex]
g_{02}=(0,-1)                      &g_{13}=(\frac{1}{4},-\frac{3}{4})& g_{24}=(-\frac{3}{4},0)          & g_{35}=(-\frac{1}{4},0)&\\[3ex]
g_{03}=(-\frac{3}{4},-\frac{3}{4})&g_{14}=(\frac{1}{4},-1)           & g_{25}=(-1,\frac{1}{4})          &                        &\\[3ex]
g_{04}=(-\frac{3}{4},-1)           &g_{15}=(0,-\frac{3}{4})          &                                  &                        &\\[3ex]
g_{05}=(-1,-\frac{3}{4})           &                                 &                                  &                        &
\end{matrix}
\end{equation*}

We only keep the $g_{ij}=(a_{ij},b_{ij})$ such that $a_{ij} b_{ij} < 0$, that is: $g_{12}$, $g_{13}$, $g_{14}$, $g_{23}$, $g_{25}$, and finally, $g_{45}$.

Identifying the set $T=\{(x_1,x_2) \mid x_1+x_2=1 \}$ with $\R$ via the projection $(x_1,x_2) \mapsto x_2$, we get that
\begin{equation*}
    H=\{  \frac{1}{5}; \frac{1}{4}; \frac{1}{2};    \frac{3}{4};  \frac{4}{5} \}.
\end{equation*}
In this setting, we need to compute the barcode template corresponding to the linear forms $u \colon \R^2 \to \R$ defined by $(x_1,x_2) \mapsto a x_1 + b x_2$ with $a> 0$, $b > 0$ and $a+b=1$. We first reduce the problem to computing the barcode of $\dR u_\ast \ms U$. For that purpose, we apply Theorem \ref{thm_proj_bar} and obtain the finite open $\Int(\Lambda)$-complex $\ms{F}=\Int(\Lambda)\{u(J)\}[-1]$ on $\R$ where
\begin{align*}
    \ms{F}^{-1}&=\k_{u(g_4)+\Int(\Lambda)} \oplus \k_{u(g_5)+\Int(\Lambda)}\\ \ms{F}^{0}&= \k_{u(g_1)+\Int(\Lambda)} \oplus \k_{u(g_2)+\Int(\Lambda)} \oplus \k_{u(g_3)+\Int(\Lambda)} \\
    \ms{F}^{1}&= \k_{u(g_0)+\Int(\Lambda)}
\end{align*}
and where the maps~$d_{\ms{F}}^{-1} \colon \ms{F}^{-1} \to \ms{F}^{0}$ and~$d_{\ms{F}}^{0} \colon \ms{F}^{0} \to \ms{F}^{1}$ have the same matrix representations as $d_{M}^{-2}$ and $d_{M}^{-1}$ respectively. We then, compute the $u(g_i)$ and get
\begin{equation*}
\begin{array}{ccccc}
u(g_0)=0 && u(g_1)= a = 1-b && u(g_2)=b \vspace{0.5cm} \\
u(g_3)=\frac{3}{4} && u(g_4)=\frac{3}{4}+ \frac{1}{4}b && u(g_5)=1- \frac{1}{4}b
\end{array}
\end{equation*}
From now on, we assume that $\k=\Z_2$. We first compute the barcode of the face $\sigma$ corresponding to the case where $b$ belongs to the interval $(\frac{1}{5},\frac{1}{4})$. This defines the following total order on the generators
\begin{equation}\label{ord:ord1}
    g_0 \prec_\sigma  g_2 \prec_\sigma g_1 \prec_\sigma g_3 \prec_\sigma g_4 \prec_\sigma g_5.
\end{equation}

Following Section \ref{subsec:qsw_conic}, the corresponding finite filtration is defined by
\begin{align*}
C(0)&=0,\\
C(i)&= \Gamma \Bigl( u(g_{i-1})+\Int(\Lambda); \, \Int(\Lambda)\{u(J)\}\Bigr)[-1],
\end{align*}
where $f_{i-1} \colon C(i-1) \to C(i)$ is induced by the restriction maps of the sheaf $\Int(\Lambda)\{u(J)\}[-1]$. One checks that this finite filtration is simplexwise.

We use the persistence algorithm to compute the barcode and refer to \cite[\textsection 3.3.1]{DW22} for the details of this algorithm. We first construct the filtered boundary matrix \eqref{mat:bnd1} associated with the above finite filtration. It is obtained by combining the differentials of the finite filtration $C(\bullet)$ which are completely determined by their values on the generators of the conic complex $\Int(\Lambda)\{u(J)\}[-1]$ (These differentials have the same matrix representation as $d^{-2}_M$ and $d^{-1}_M$). The entries of the filtered boundary matrix are ordered according to the total order \eqref{ord:ord1} on the generators. This yields the following:
\begin{equation} \label{mat:bnd1}
     \bordermatrix{ & u(g_2) & u(g_1) & u(g_3) & u(g_4) & u(g_5) \cr
       u(g_0) & 1 & 1 & 1 & 0 & 0 \cr
       u(g_2) & 0 & 0 & 0 &1 & 1 \cr
       u(g_1) & 0 & 0 & 0 & 1 & 0 \cr
       u(g_3) & 0 & 0 & 0 & 0 & 1} \qquad
\end{equation}
Recall that, that the boundary matrix $B$ of the cochain complex of persistence modules that is associated to a conic complex $\ms G = \Int(\Lambda)\{J\}$ is such that $B_{i,j} = 1$ if, and only if, there exists $k \in \Z$ such that $j \in J_k$ and $i \in J_{k+1}$ with $d^{k}_\ms G|_{\k_{(i,+\infty)}} \neq 0$.
Applying the reduction process, we get the matrix:
\begin{equation*}
     \bordermatrix{ & u(g_2) & u(g_1) & u(g_3) & u(g_4) & u(g_5) \cr
       u(g_0) & 1 & 0 & 0 & 0 & 0 \cr
       u(g_2) & 0 & 0 & 0 &1 & 1 \cr
       u(g_1) & 0 & 0 & 0 & 1 & 0 \cr
       u(g_3) & 0 & 0 & 0 & 0 & 1}\qquad 
\end{equation*}
This implies that the projected barcode functor (see Equation \eqref{map:projbar}) evaluated on the pair $(u,\ms U)$ with $u$ as above is given by
\begin{equation*}
    \ms P^\gamma (u, \ms U) \simeq \k_{(u(g_0);u(g_2)]} \oplus \k_{(u(g_1);u(g_4)]}[-1] \oplus \k_{(u(g_3);u(g_5)]}[-1].
\end{equation*}

Second, we compute the barcode of the case $b=\frac{1}{4}$. We compute the $u(g_i)$ and get
\begin{equation*}
\begin{array}{ccccc}
u(g_0)=0 && u(g_1)= \frac{3}{4} && u(g_2)=\frac{1}{4} \vspace{0.5cm} \\
u(g_3)=\frac{3}{4} && u(g_4)=\frac{13}{16} && u(g_5)=\frac{15}{16}
\end{array}
\end{equation*}
This defines the following total pre-order on the generators:
\begin{align*}
u(g_0)  < u(g_2) <   u(g_1)=u(g_3) < u(g_4) < u(g_5)
\end{align*}
In this setting, we compute the barcode corresponding to the linear form 
\[
u \colon \R^2 \to \R \textnormal{ defined by } (x_1,x_2) \mapsto \frac{3}{4} x_1 + \frac{1}{4} x_2.
\]
We apply Theorem \ref{thm_proj_bar} and obtain the complex $\Int(\Lambda)\{u(J)\}[-1]$ on $\R$
\begin{align*}
    \k_{\frac{13}{16}+\Int(\Lambda)} \oplus \k_{\frac{15}{16}+\Int(\Lambda)} \longrightarrow \k_{\frac{3}{4}+\Int(\Lambda)} \oplus \k_{\frac{1}{4}+\Int(\Lambda)} \oplus \k_{\frac{3}{4}+\Int(\Lambda)} \longrightarrow\k_{\Int(\Lambda)},
\end{align*}
the non-zero terms of which are in cohomological degrees $-1$, 0 and 1. Again, following Section \ref{subsec:qsw_conic}, we obtain the corresponding finite filtration defined by
\begin{align*}
C(0)&=0,\\
C(i)&= \Gamma \Bigl( u(g_{i-1})+\Int(\Lambda); \, \Int(\Lambda)\{u(J)\}\Bigr)[-1].
\end{align*}
This filtration is not simplex-wise, so we now construct a simplex-wise finite filtration applying Section \ref{sec_qsimplexwise} and use the indexing convention described ibid. We index the intermediate complex introduced to construct the simplex-wise filtration $\lbrack C(\bullet)\rbrack_{\mathrm{sw}}$ by choosing the middle of the open interval~$(3, 4)$. This yields the following:  

\begin{equation} \label{filt:simplexwise}
    \begin{tikzcd}[row sep=0.3em]
        \!\!\!\lbrack C(0)\rbrack_{\mathrm{sw}}=  0 &                        &\\
        \lbrack C(1)\rbrack_{\mathrm{sw}}=  \k_0&                        & \\
        \lbrack C(2)\rbrack_{\mathrm{sw}}= \k_0 & \ar[l] \k_2                    &\\
        \lbrack C(3)\rbrack_{\mathrm{sw}}= \k_0 & \ar[l] \k_1 \oplus \k_2           &\\
        \lbrack C(\frac{7}{2})\rbrack_{\mathrm{sw}}= \k_0 & \ar[l] \k_1 \oplus \k_2 \oplus \k_3 &\\
        \lbrack C(4)\rbrack_{\mathrm{sw}}= \k_0 & \ar[l] \k_1 \oplus \k_2 \oplus \k_3 & \ar[l] \k_4\\
        \lbrack C(5)\rbrack_{\mathrm{sw}}= \k_0 & \ar[l] \k_1 \oplus \k_2 \oplus \k_3 & \ar[l] \k_4 \oplus \k_5
    \end{tikzcd}
\end{equation}
where $\k_i$ corresponds to the copy of the field $\k$ associated with the generator $g_i$.

We  construct the filtered boundary matrix \eqref{mat:bnd2} associated with the simplexwise finite filtration \eqref{filt:simplexwise}. It is obtained by combining the differentials

It is obtained by combining the differentials of the finite filtration \eqref{filt:simplexwise} which are completely determined by their values on the generators of the conic complex $\Int(\Lambda)\{u(J)\}[-1]$ (these differentials have the same matrix representation as $d^{-2}_M$ and $d^{-1}_M$). The entries of the filtered boundary matrix are ordered according to the order provided by the simplex-wise filtration $\lbrack C(\bullet)\rbrack_{\mathrm{sw}}$ obtained via procedure described in Section \ref{sec_qsimplexwise}.

We have the following boundary matrix:
\begin{equation} \label{mat:bnd2}
     \bordermatrix{ & 2 & 3 & \frac{7}{2} & 4 & 5 \cr
       1 & 1 & 1 & 1 & 0 & 0 \cr
       2 & 0 & 0 & 0 &1 & 1 \cr
       3 & 0 & 0 & 0 & 1 & 0 \cr
       \frac{7}{2} & 0 & 0 & 0 & 0 & 1} \qquad
\end{equation}
Applying the reduction process, we get once again the matrix:
\begin{equation*}
     \bordermatrix{ & 2 & 3 & \frac{7}{2} & 4 & 5 \cr
       1 & 1 & 0 & 0 & 0 & 0 \cr
       2 & 0 & 0 & 0 &1 & 1 \cr
       3 & 0 & 0 & 0 & 1 & 0 \cr
       \frac{7}{2} & 0 & 0 & 0 & 0 & 1}\qquad 
\end{equation*}
Hence, we get the following barcode:
\begin{equation*}
    \B'=\{ [1,2); \; [3,4); \; [ \frac{7}{2}, 5) \}
\end{equation*}
Applying Proposition \ref{prop_finite_simplexwise_barcode} yields:
\begin{equation*}
    \B=\{ [1,2); \; [3,4); \; [3, 5) \}
\end{equation*}
Finally, applying Theorem \ref{thm_red_qsw}, we get that the projected barcode is
\begin{equation*}
    \ms P^\gamma (u, \ms U) \simeq \k_{(u(g_0);u(g_2)]} \oplus \k_{(u(g_1);u(g_4)]}[-1] \oplus \k_{(u(g_1);u(g_5)]}[-1].
\end{equation*}

The barcode templates for the faces of $H$ are computed in a similar manner to the face $\sigma$ case, where $ b \in \left(\frac{1}{5}, \frac{1}{4}\right)$. For the zero-dimensional faces determined by an element of $ H $, the computation is similar to the case $ b = \frac{1}{4} $. Therefore, these calculations are left to the reader.

\section{Experiments}
\label{sec:experiments}

\subsection{Setup}

The algorithm has been implemented in Python3 in the 2-parameter setting (Section \ref{sec:algo_2-params}), and can compute projected barcodes for any scc2020 formated free resolutions \cite{KL21}. These resolutions can for instance be computed with 2pac~\cite{BLL23}. Barcode computations are done using the PHAT library~\cite{bauer2014phat}, based on the {\em twist} algorithm. The optimization using vineyards \cite{CEM06} has not been implemented.

We ran our code on a Dell Precision 3460 with an 13th Gen Intel Core i9-13900 x 32 chip and 64~GB of RAM. The code was run with Python 3.10.12. Computation times were measured with \emph{time} from the \emph{time} module.

\subsection{Results}

Our results were obtained on a sample of 30 points sampled uniformly on the right half of a circle of radius one, 40 points sampled uniformly on a circle of unit radius plus one outlier. Calling $P$ the resulting point sample with 71 points, we performed a density-Rips bifiltration with density given by the function $x \mapsto | \{ y  \in P \mid |x-y| < 1.5 \} |$ (see Figure \ref{fig:halfcircle}).

\begin{figure}[ht]
    \centering
    \includegraphics[scale = 0.5]{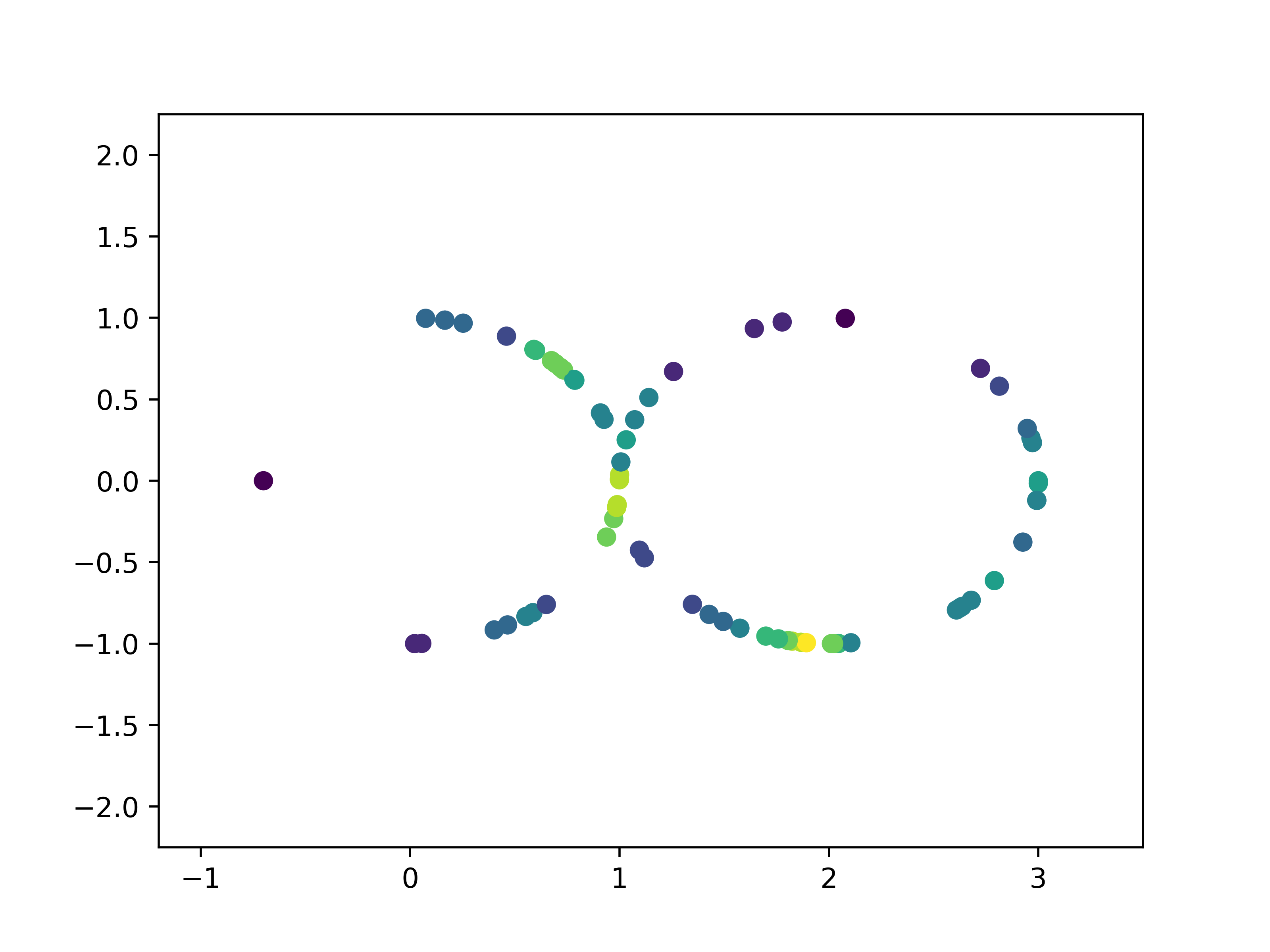}
    \caption{The point sample $P$.}
    \label{fig:halfcircle}
\end{figure}
 
Figure~\ref{fig:res} shows the  projected barcode of the simplicial homology in degree~1 of the function-Rips bifiltration, queried along the linear forms  $(x,y) \mapsto 0.875x + 0.125y$ and $(x,y) \mapsto 1x$. As intuition suggests, the first linear form, which takes the density into account, allows one to recover the generators of the circle in degree~1 homology while minimising the effect of the outlier; meanwhile the second linear form, which corresponds to the Rips filtration, shows two generators due to the outlier's position.\\

\begin{figure}[htb]
\centering
\includegraphics[scale = 0.3]{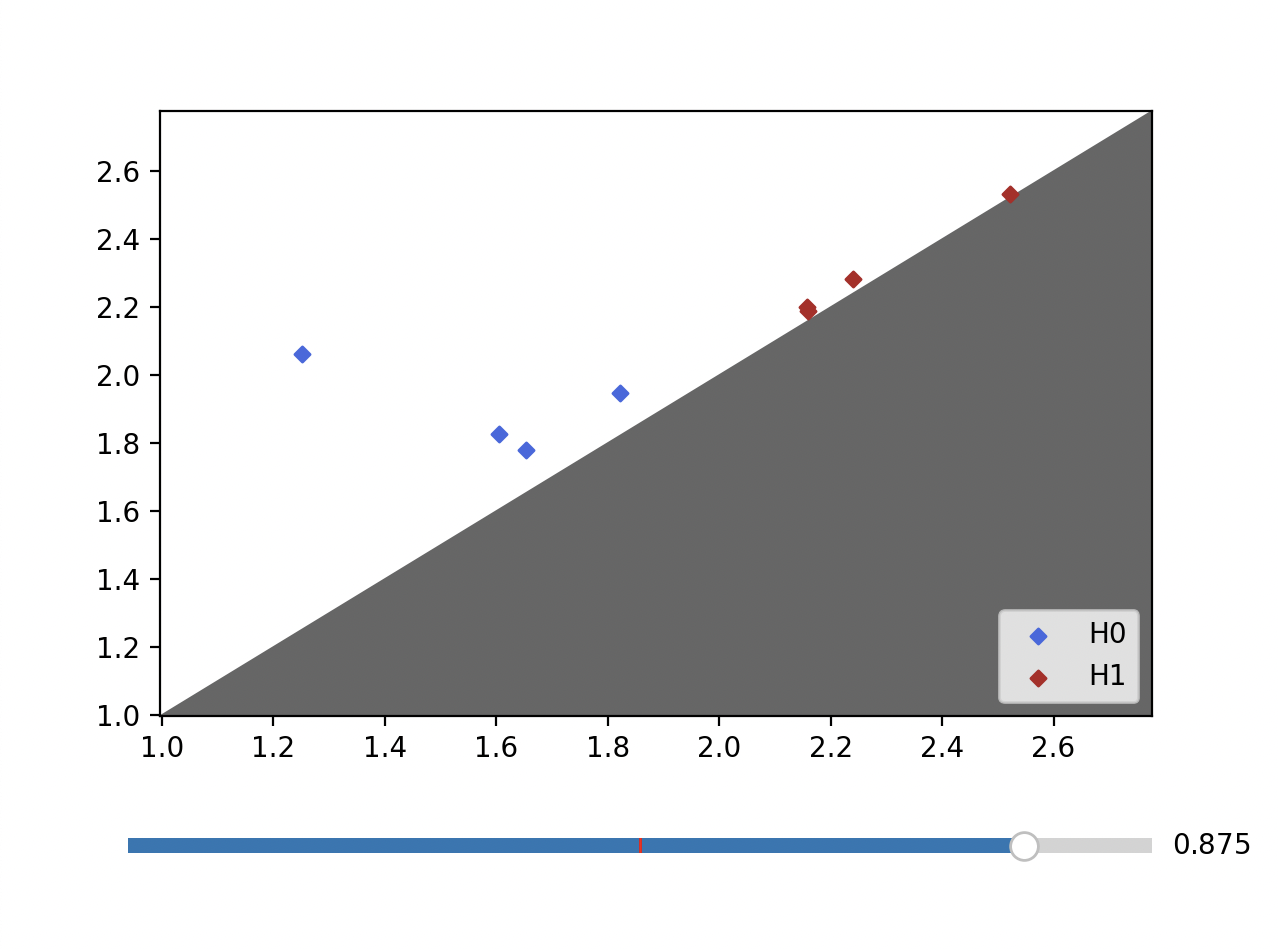}
\includegraphics[scale = 0.3]{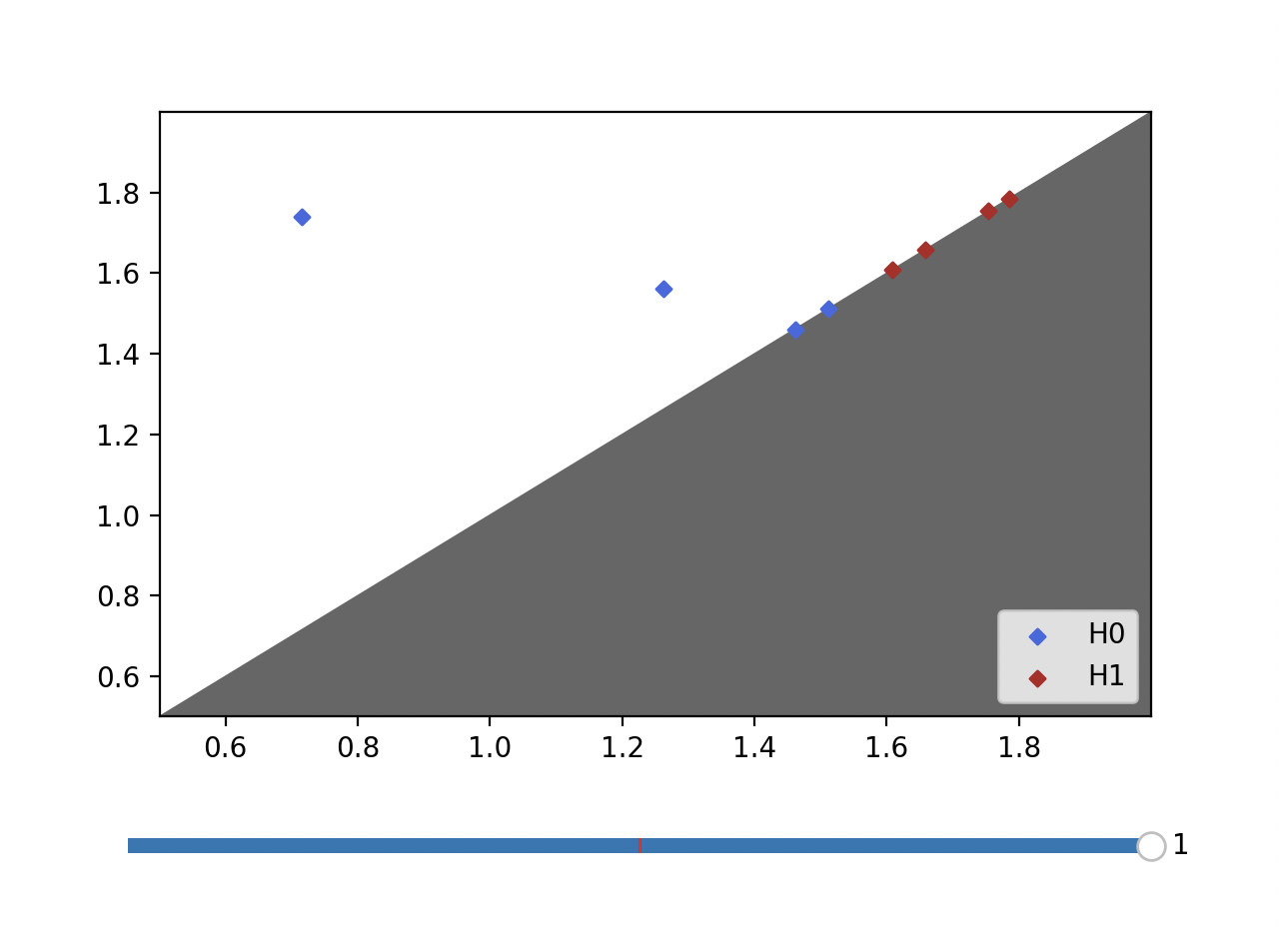}
\caption{The projected barcodes of the function-Rips bifiltration obtained by pushforwards along the linear forms $(x,y) \mapsto 0.875x+0.125y$ (left) and $(x,y) \mapsto 1x$ (right). Blue points represent persistence pairs with creator in degree 0 in the resolution of the first simplicial homology group, while red points represent persistence pairs with creator in degree 1 in the same resolution.}
\label{fig:res}
\end{figure}

Table \ref{tab:perf} reports the computation times of the projected barcode template, and of queries along a linear form.

\begin{table}[ht]
\begin{tabular}{|p{2.5cm}||>{\raggedleft\arraybackslash}p{2cm}|>{\raggedleft\arraybackslash}p{1.7cm}|>{\raggedleft\arraybackslash}p{2cm}|>{\raggedleft\arraybackslash}p{2cm}|>{\raggedleft\arraybackslash}p{2cm}|}
 \hline
 \multicolumn{6}{|c|}{Performances} \\
 \hline
 \raggedleft Sample name & Number of generators & Number of faces & Query with \quad template & Query without template & Projected Barcode Template\\
 \hline
 Circle $\H_0$ &  332 &  16602 &  0.341 ms &  0.978 ms & { 11.068 s} \\
 Circle $\H_1$ &  72 &  1443 &  0.093 ms &  0.175 ms & {0.422 s}\\
 \hline
 Torus $\H_0$ &   1516 &  284346 &  1.904 ms &  3.567 ms & { 1072.477 s}\\
 Torus $\H_1$ &   782 &  163034 &  0.810 ms &  1.904 ms & { 308.611 s}\\
 Torus $\H_2$ &  96 &  1791 &  0.102 ms  &  0.237 ms & { 0.355 s}\\
 \hline
 Octogone $\H_0$ &  39 &  106 &  0.966 ms &  0.120 ms & { 6.699 s}\\
 Octogone $\H_1$ &  8 &  7 &  0.053 ms  &  0.112 ms & { 1.571 s}\\
 \hline
 Dragon $\H_0$ &  1082 &  211000 &  1.077 ms &  2.592 ms & { 502.068 s} \\
 Dragon $\H_1$ &  824 &  219327 &  0.818 ms &  1.963 ms & { 467.449 s} \\
 Dragon $\H_2$ &  120 &  4928 &  0.172 ms &  0.289 ms & { 1.132 s }\\
\hline
\end{tabular}
\caption{The sample set is composed of : 100 points sampled on a circle, 500 points sampled on a torus both with a vertical $\chi^2$ offset perturbation and 300 sampled on the Stanford Dragon available in \cite{OPTGH17}; these where analysed using a density Rips filtration with density estimated by the number of neighbours within a ball with specified radius. We also included a multicritical example with a degree-Rips bilfitration on an octogone with two outliers.}
\label{tab:perf}
\end{table}



\clearpage

\bibliographystyle{plainurl}
\bibliography{bibliography}

\begin{thebibliography}{10}

\bibitem{bauer2014phat}
Ulrich Bauer, Michael Kerber, Jan Reininghaus, and Hubert Wagner.
\newblock Phat--persistent homology algorithms toolbox.
\newblock In {\em Mathematical Software--ICMS 2014}, pages 137--143. Springer, 2014.
\newblock See {\footnotesize\url{http://github.com/blazs/phat}}.

\bibitem{BLL23}
Ulrich Bauer, Fabian Lenzen, and Michael Lesnick.
\newblock {Efficient Two-Parameter Persistence Computation via Cohomology}.
\newblock In Erin~W. Chambers and Joachim Gudmundsson, editors, {\em 39th International Symposium on Computational Geometry (SoCG 2023)}, volume 258 of {\em Leibniz International Proceedings in Informatics (LIPIcs)}, pages 15:1--15:17, Dagstuhl, Germany, 2023. Schloss Dagstuhl -- Leibniz-Zentrum f{\"u}r Informatik.

\bibitem{BG22}
Nicolas Berkouk and Grégory Ginot.
\newblock A derived isometry theorem for sheaves.
\newblock {\em Advances in Mathematics}, 394:108033, 2022.

\bibitem{BP23}
Nicolas Berkouk and Francois Petit.
\newblock Projected distances for multi-parameter persistence modules, 02 2023.
\newblock URL: \url{https://arxiv.org/abs/2206.08818}.

\bibitem{BP21}
Nicolas Berkouk and François Petit.
\newblock Ephemeral persistence modules and distance comparison.
\newblock {\em Algebraic \& Geometric Topology}, 21(1):247–277, 2021.

\bibitem{BM88}
Edward Bierstone and Pierre~D. Milman.
\newblock Semianalytic and subanalytic sets.
\newblock {\em Publications mathématiques de l’IHÉS}, 67(1):5–42, Jan 1988.

\bibitem{BK21}
H{\aa}vard~Bakke Bjerkevik and Michael Kerber.
\newblock Asymptotic improvements on the exact matching distance for 2-parameter persistence, 2021.
\newblock URL: \url{https://arxiv.org/abs/2111.10303}.

\bibitem{CFFFL13}
Andrea Cerri, Barbara~Di Fabio, Massimo Ferri, Patrizio Frosini, and Claudia Landi.
\newblock Betti numbers in multidimensional persistent homology are stable functions.
\newblock {\em Mathematical Methods in the Applied Sciences}, 36(12):1543–1557, 2013.

\bibitem{CCBS16}
Frédéric Chazal, William Crawley-Boevey, and Vin de~Silva.
\newblock The observable structure of persistence modules.
\newblock {\em Homology, Homotopy and Applications}, 18(2):247–265, 2016.

\bibitem{CEM06}
David Cohen-Steiner, Herbert Edelsbrunner, and Dmitriy Morozov.
\newblock Vines and vineyards by updating persistence in linear time.
\newblock {\em Proceedings of the twenty-second annual symposium on Computational geometry}, 2006.

\bibitem{C15}
William Crawley-Boevey.
\newblock Decomposition of pointwise finite-dimensional persistence modules.
\newblock {\em Journal of Algebra and Its Applications}, 14(05):1550066, 2015.

\bibitem{Curry}
J~Curry.
\newblock Sheaves, cosheaves and applications, 2014.
\newblock URL: \url{https://arxiv.org/abs/1303.3255}.

\bibitem{DW22}
Tamal~Krishna Dey and Yusu Wang.
\newblock {\em Computational topology for Data Analysis}.
\newblock Cambridge University Press, 2022.

\bibitem{ezra2020decomposing}
Esther Ezra, Sariel Har-Peled, Haim Kaplan, and Micha Sharir.
\newblock Decomposing arrangements of hyperplanes: {VC}-dimension, combinatorial dimension, and point location.
\newblock {\em Discrete \& Computational Geometry}, 64(1):109--173, 2020.

\bibitem{KS90}
M.~Kashiwara and P.~Schapira.
\newblock {\em Sheaves on manifolds}, volume 292 of {\em Grundlehren der Mathematischen Wissenschaften}.
\newblock Springer-Verlag, Berlin, 1990.
\newblock With a chapter in French by Christian Houzel.

\bibitem{KS6}
Masaki Kashiwara and Pierre Schapira.
\newblock {\em Categories and Sheaves}, volume 332 of {\em Grundlehren der mathematischen Wissenschaften}.
\newblock Springer Berlin Heidelberg, 2006.

\bibitem{KS18}
Masaki Kashiwara and Pierre Schapira.
\newblock Persistent homology and microlocal sheaf theory.
\newblock {\em Journal of Applied and Computational Topology}, 2(1-2):83–113, 2018.

\bibitem{KL21}
M.~Kerber and M.~Lesnick.
\newblock scc2020: A file format for sparse chain complexes in tda, 2021.
\newblock URL: \url{https://bitbucket.org/mkerber/chain_complex_format}.

\bibitem{KLO20}
Michael Kerber, Michael Lesnick, and Steve Oudot.
\newblock Exact computation of the matching distance on 2-parameter persistence modules.
\newblock {\em Journal of Computational Geometry}, 11(2), 2020.

\bibitem{LW15}
Michael Lesnick and Matthew Wright.
\newblock Interactive visualization of 2-d persistence modules, 2015.
\newblock URL: \url{https://arxiv.org/abs/1512.00180}.

\bibitem{M93}
S.~Meiser.
\newblock Point location in arrangements of hyperplanes.
\newblock {\em Information and Computation}, 106(2):286–303, 1993.

\bibitem{OPTGH17}
Nina Otter, Mason~A Porter, Ulrike Tillmann, Peter Grindrod, and Heather~A Harrington.
\newblock A roadmap for the computation of persistent homology.
\newblock {\em EPJ Data Science}, 6(1), Aug 2017.
\newblock \href {https://doi.org/10.1140/epjds/s13688-017-0109-5} {\path{doi:10.1140/epjds/s13688-017-0109-5}}.

\bibitem{S23}
Pierre Schapira.
\newblock Constructible sheaves and functions up to infinity.
\newblock {\em Journal of Applied and Computational Topology}, 2023.

\bibitem{PS21}
Pierre Shapira and Francois Petit.
\newblock Thickening of the diagonal and interleaving distance, 06 2021.
\newblock URL: \url{https://arxiv.org/abs/2006.13150}.

\bibitem{V20}
Oliver Vipond.
\newblock Local equivalence of metrics for multiparameter persistence modules, 2020.
\newblock URL: \url{https://arxiv.org/abs/2004.11926}.

\end{thebibliography}

\appendix


\section{Technical supplement}

\label{sec:appendix_sec4}


\subsection{Derived functors between derived categories}

Let $\mc C, \mc D$ be two abelian categories, and let $F : \mc C \to \mc D$ be an additive functor. The derived category $\D^+(\mc C)$ of $\mc C$ is the localization of the category of cochain complexes $\Ch^+(\mc C)$ with respect to the quasi-isomorphisms. Under suitable conditions, the functor $F$ induces a derived functor $\dR F \colon \D^+(\mc C) \to \D^+(\mc D)$. We refer the reader to \cite[Chapter I]{KS90} and especially to sections 1.7 and 1.8 therein for details. The aim of this section is to briefly recall a method for the evaluation of a derived functor on an object. This technique is well-known and implicit in \textit{ibid}.

\noindent A full additive subcategory $\mc{J}$ of $\mc{C}$ is called \emph{$F$-injective} if the three following conditions hold :
\begin{itemize}
\item for any object $X$ of $\mc C$, there is a map $X \to J$, with $J$ an object of $\mc J$, such that $0 \to X \to J$ is exact;
\item the category $\mc J$ is closed under taking cokernels;
\item the functor $F$ sends short exact sequences in $\mc J$ to short exact sequences.
\end{itemize}
Assume that $F$ is left-exact and that $\mc C$ has enough injectives. An object $X$ of $\mc C$ is \emph{$F$-acyclic} if $\dR^nF(X) = 0$ for $n \neq 0$. The full subcategory of $F$-acyclic objects is $F$-injective \cite[exercice I.19]{KS90} and it contains the injective objects. Then, the next folklore lemma follows from \cite[Proposition 1.8.3]{KS90}:

\begin{lemma}
\label{lem_acyclic}
Let $F : \mc C \to \mc D$ be a left exact functor between abelian categories. Suppose that $\mc C$ has enough injectives, and let $X$ be an object of $\Ch^+(\mc C)$.
If for every integer $n$ the cochain complex $X^n$ (concentrated in degree 0) is $F$-acyclic, then, in $\D^+(\mc C)$, $\dR F (X)$ is isomorphic to
\[\begin{tikzcd}
	\cdots & F(X^{-1}) &  F(X^0) & F(X^1) & \cdots
	\arrow["F(d^0)", from=1-3, to=1-4]
	\arrow["F(d^1)", from=1-4, to=1-5]
	\arrow["F(d^{-1})", from=1-2, to=1-3]
	\arrow["F(d^{-2})", from=1-1, to=1-2]
\end{tikzcd}\]
\end{lemma}

\subsection{Verdier duality}\label{app:duality}

Let $M$ be an $n$-dimensional real manifold. For any sheaf $\ms F$ in $\D^b(\k_M)$, we recall the notation introduced in \cite[ Definition 3.1.16]{KS90}: \[ \dual_M \ms F = \dR \HOM ( \ms F, \omega_M ), \quad \quad \dual'_M \ms F = \dR \HOM (\ms F, \k_M), \] where the dualizing complex is $\omega_M := a_M^! \k$ for $a_M: M \to *$.
In \cite[section 3.4]{KS90}, the authors introduce the notion of \emph{cohomologically constructible sheaves}. Here, only the following property will be used \cite[Proposition 3.4.3.]{KS90} : if $\ms F$ is a cohomologically constructible sheaf then $\dual_M \ms F$ is cohomologically constructible and $\ms F \to \dual_M \dual_M \ms F$ is an isomorphism.
It follows from \cite[Exercice III.4]{KS90} that, if an open subset $U\subseteq M$ is locally cohomologically trivial, then $\k_U$ and $\k_{\overline{U}}$ are cohomologically constructible. By \emph{locally cohomologically trivial (or l.c.t.)}, we mean that $ \big( \dR \Gamma_{\overline U} \hspace{4pt} \k_M \big)_x \simeq 0 \text{ and } \big( \dR \Gamma_{U} \hspace{4pt} \k_M \big)_x \simeq \k$ for every $x \in \overline U \setminus U$. It follows from \cite[Exercice III.4]{KS90} that, if $U$ is locally cohomologically trivial, then
\[\dual^\prime_M \k_U \simeq \k_{\overline{U}}, \quad \quad \dual_M' \k_{\overline{U}} \simeq \k_U.\]

An open set $V$ of $M$ is called \emph{locally topologically convex (or l.t.c.)} if every point in $V$ has an open neighborhood that is homeomorphic to an open convex subset of a real vector space. In particular, if $V$ l.t.c then it is l.c.t (see \cite[section 2.1]{PS21}).

\section{Details of the procedure of Section~\ref{sec_qsimplexwise}}\label{sec:appendix_proc}

Here are now the details of the procedure sketched in Section~\ref{sec_qsimplexwise}. Let $j_{\max}$ be the degree of the maximal non-zero term in~$C(i)$, which by assumption is a bounded cochain complex. We proceed as follows: 

\begin{itemize}
    \item At step~0, we map $C(i-1)$ to $\im(f_{i-1})$ via $f_{i-1}$. We call $f_{i-1,0}$ the corresponding morphism. Note that its cokernel is trivial, so it is not per se a simplex insertion. To get a true simplex insertion it is sufficient to merge this step with the next one in the final procedure.
\end{itemize}

\begin{itemize}
    \item At step~1, we add $N_{j_{\max}}$ in degree~$j_{\max}$ to our complex, via the following morphism called~$f_{i-1,1}$:
\[
\begin{tikzcd}
\cdots \ar[r] & \im(f_{i-1}^{j_{\max}-1}) \ar[r, "d_i^{j_{\max}-1}"] 	\ar[d, shift left=1, no head] \ar[d, shift right=1, no head] & \im(f_{i-1}^{j_{\max}}) \ar[r, "d_i^{j_{\max}}"] 	\ar[d, tail, "{\left[\begin{smallmatrix}\id\\0\end{smallmatrix}\right]}"] & 0 \ar[r] \ar[d, shift left=1, no head] \ar[d, shift right=1, no head] & \cdots\\[5pt]
\cdots \ar[r] & \im(f_{i-1}^{j_{\max}-1}) \ar[r, "d_i^{j_{\max}-1}"] & \im(f_{i-1}^{j_{\max}}) \oplus N_{j_{\max}} \ar[r, "d_i^{j_{\max}}"] &  0 \ar[r] & \cdots
\end{tikzcd} 
\]
Strictly speaking, the horizontal maps in the diagram are the restrictions of the differentials of~$C(i)$ to the appropriate spaces. Note that we have both $\im (f_{i-1}^{j_{\max}+1})=0$ and $d_i^{j_{\max}}=0$ here, because $C(i)^{j_{\max}+1}=0$. The two rows are then well-defined complexes and the vertical arrows form a well-defined monomorphism of complexes between them.
\end{itemize}

\begin{itemize}
    \item At every subsequent step s>1, we add $N_{j}$ in degree~$j=j_{\max}+1-s$ to our complex, via the following morphism called~$f_{i-1,s}$:
\[\begin{tikzcd}
\cdots \ar[r] & \im(f_{i-1}^{j-1}) \ar[r, "d_i^{j-1}"] 	\ar[d, shift left=1, no head] \ar[d, shift right=1, no head] & \im(f_{i-1}^{j}) \ar[r, "d_i^{j}"] 	\ar[d, tail, "{\left[\begin{smallmatrix}\id\\0\end{smallmatrix}\right]}"] & \im(f_{i-1}^{j+1}) \oplus N_{j+1} \ar[r] \ar[d, shift left=1, no head] \ar[d, shift right=1, no head] & \cdots\\[5pt]
\cdots \ar[r] & \im(f_{i-1}^{j-1}) \ar[r, "d_i^{j-1}"] & \im(f_{i-1}^{j}) \oplus N_{j} \ar[r, "d_i^{j}"] &  \im(f_{i-1}^{j+1}) \oplus N_{j+1} \ar[r] & \cdots
\end{tikzcd}\]    
Again, strictly speaking, the horizontal maps in the diagram are the restrictions of the differentials of~$C(i)$ to the appropriate spaces. Note in particular that $d_i^{j}$ makes sense here because $N_{j+1}$ has been inserted previously. The two rows are then well-defined complexes, and the vertical arrows form a well-defined monomorphism of complexes between them.
\end{itemize}

After finitely many steps (because $C(i)$ is bounded), our current complex becomes~$C(i)$, and the  intermediate morphisms $f_{i-1,s}$ compose to~$f_{i-1}$. Note however that they may not be simplex insertions individually, because each vector space complement~$N_{j_{\max}+1-s}$ is inserted at once via~$f_{i-1,s}$. We therefore further decompose every~$f_{i-1,s}$ into a sequence of simplex insertions, each inserting a new dimension of~$N_{j_{\max}+1-s}$ in the complex in degree~$(j_{\max}+1-s)$---this requires fixing an arbitrary ordered basis of~$N_{j_{\max}+1-s}$ in advance, the choice of ordered basis being irrelevant.

\end{document}